\documentclass[11pt,letterpaper]{amsart}
\newtheorem{theoreme}{Theorem}[section]
\newtheorem{fait}[theoreme]{Fact}
\newtheorem{lemme}[theoreme]{Lemma}

\newtheorem{proposition}[theoreme]{Proposition}

\newtheorem{remarque}[theoreme]{Remark}

\newtheorem{corollaire}[theoreme]{Corollary}

\usepackage{xy}
\xyoption{all}
\usepackage[normalem]{ulem}
\usepackage{enumerate}\usepackage{color}
\usepackage{graphicx,epsfig}
\usepackage{amscd}
\usepackage{amssymb}
\usepackage[shortlabels]{enumitem}
\usepackage{ulem}
\usepackage{tikz-cd}
\usepackage{hyperref}

\usepackage[utf8]{inputenc}
\usepackage[T1]{fontenc}
\usepackage{lmodern}
\usepackage{tensor}
\usepackage{colonequals}
\usepackage{amsxtra}
\usepackage{xcolor}

\hypersetup{colorlinks,allcolors=blue}

     \newcommand{\Can}{\hbox{\rm \tiny an}}
     \newcommand{\semis}{\hbox{\rm \tiny ss}}
     \newcommand{\red}{\hbox{\rm \tiny red}}
     \newcommand{\ad}{\hbox{\rm \tiny ad}}
    \newcommand{\der}{\hbox{\rm \tiny der}}
   \newcommand{\geo}{\hbox{\rm \tiny geo}}
      
   \newcommand{\op}{\hbox{\rm \tiny op}}

               \newcommand{\ULA}{{\hbox{\rm \tiny ULA}}}
             
                \newcommand{\liss}{{\hbox{\rm \tiny liss}}}

  \newcommand{\Sim}{\hbox{\rm im}}

  \newcommand{\Sdim}{\hbox{\rm dim}}

 \newcommand{\SH}{\hbox{\rm H}}
\newcommand{\IC}{\hbox{\rm IC}}
  
  \newcommand{\SGL}{\hbox{\rm GL}}

      \newcommand{\SHom}{\hbox{\rm Hom}}
         
           \newcommand{\SEnd}{\hbox{\rm End}}
            
                         \newcommand{\SLie}{\hbox{\rm Lie}}

                    \newcommand{\spec}{\hbox{\rm spec}}
                    
                      \newcommand{\Rep}{\hbox{\rm Rep}}
                         \newcommand{\Perv}{\hbox{\rm Perv}}
                             \newcommand{\Vect}{\hbox{\rm Vect}}
                             
                          \newcommand{\Loc}{\hbox{\rm Loc}}

     \newcommand{\C}{\mathbb{C}}
    
      \newcommand{\G}{\mathbb{G}}
         \newcommand{\Q}{\mathbb{Q}}
           
     \newcommand{\Z}{\mathbb{Z}}

\newcommand{\cA}{\mathcal{A}}
\newcommand{\cD}{\mathcal{D}}

\newcommand{\cS}{\mathcal{S}}

\newcommand{\cN}{\mathcal{N}}

\newcommand{\cL}{\mathcal{L}}
\newcommand{\cK}{\mathcal{K}}
\newcommand{\cC}{\mathcal{C}}
\newcommand{\cF}{\mathcal{F}}

\newcommand{\cO}{\mathcal{O}}
\newcommand{\cY}{\mathcal{Y}}
\newcommand{\cZ}{\mathcal{Z}}
\newcommand{\cP}{\mathcal{P}}
\newcommand{\cQ}{\mathcal{Q}}
\newcommand{\cU}{\mathcal{U}}
\newcommand{\cV}{\mathcal{V}}
\newcommand{\cT}{\mathcal{T}}
\newcommand{\cX}{\mathcal{X}}

\newcommand{\Spec}{\mathrm{Spec}}

\newcommand{\Gal}{\mathrm{Gal}}

\newcommand{\cpt}{\mathrm{cpt}}

\title{Variation of Tannaka groups of perverse sheaves in family}

\author{Anna Cadoret and Haohao Liu}

\date{\today}
\begin{document}
\maketitle

\begin{abstract}  Let $k$ be a field of characteristic $0$, let $S$ be a smooth, geometrically connected variety over $k$, with generic point $\eta$, and $f:\cX\rightarrow S$ a morphism separated and of finite type. Fix a  prime $\ell$. Let  $\cP$ be an $f$-universally locally acyclic relative perverse $\overline{\Q}_\ell$-sheaf  on $\cX/S$. We prove that if for some (equivalently, every) geometric point $\bar \eta$ over $\eta$ the restriction $\cP|_{\cX_{\bar \eta}}$ is simple as a perverse $\overline{\Q}_\ell$-sheaf on $\cX_{\bar \eta}$, then there is a non-empty open subscheme $U\subset S$ such that, for every geometric point $\bar s$ on  $U$, the restriction $\cP|_{\cX_{\bar s}}$ is  simple as a perverse $\overline{\Q}_\ell$-sheaf on $\cX_{\bar s}$. When  $f:\cX\rightarrow S$ is an abelian scheme, we give applications of this result to the variation with $s\in S$ of the Tannaka group of $\cP|_{\cX_{\bar s}}$.
 
 \end{abstract}

\DeclareRobustCommand{\SkipTocEntry}[5]{}
\setcounter{tocdepth}{2} 
\tableofcontents

\textit{}\\
\section{Introduction}
\noindent Let $k$ be a field of characteristic $0$, let $S$ be a smooth, geometrically connected variety over $k$, with generic point $\eta$, and let $g:\cY\rightarrow S$ be a smooth projective $S$-scheme of relative dimension $d$. A general and central question in algebraic geometry is to understand how the fibers $\cY_s$ vary with $s\in S$. For instance, one may ask when some power $\cY_{\bar s}^n$  of $\cY_{\bar s}$ carries exceptional algebraic cycles. If $k\subset \C$, the Hodge conjecture predicts that this is the same as asking when the Mumford Tate group  $G(\cV)_s$ of the polarizable $\Q$-Hodge structure $\SH^\bullet(\cY^{\Can}_{s},\Q)\simeq s^*R^\bullet g^{\Can}_*\Q$  becomes smaller than the generic Mumford-Tate group $G(\cV)$ of the  polarizable $\Q$-variation of Hodge structures $\cV:=R^\bullet g^{\Can}_*\Q $ on the analytification $S^{\Can}$ of  $S\times_k\C$. Here "becomes smaller" makes sense because the Tannaka categories of  polarizable $\Q$-variations of Hodge structures are functorial with respect to pullbacks along morphism of complex analytic spaces $S'\rightarrow S$ so that one can view naturally $G(\cV)_s$ as a subgroup of $G(\cV)$. This leads to introduce and study the Hodge locus $$S_{\cV}:=\lbrace s\in S\;|\; G(\cV)_s\subsetneq G(\cV)\rbrace $$ of a  polarizable $\Q$-variation of Hodge structures $\cV$ on $S^{\Can}$. If $k$ is finitely generated over $\Q$, similar considerations apply with  $\ell$-adic \'etale local systems on $S$ yielding the introduction of the Tate locus
 $S_{\cV_\ell}\subset S$ of such a $\Q_\ell$-local system $\cV_\ell$. Under mild assumptions, general heuristics predict that these exceptional loci $S_{\cV}$, $S_{\cV_\ell}$ are sparse in some precise sense - \textit{e.g.} that the atypical part of the Hodge locus $S_\cV$ is not Zariski-dense in $S$ \cite{KlinglerICM} or that the set of $k$-rational points in the Tate locus $S_{\cV_\ell}$ is not Zariski-dense in $S$ \cite{Bas}. Proving such sparsity results is notoriously challenging as it requires constructing bridges between the Zariski topology of $S$ and the analytic natures of the coefficients $\cV$, $\cV_\ell$. The results of this article are also partly motivated  by the problem of understanding  how the fibers $\cY_s$ vary with $s\in S$ and inspired by the above Tannaka approaches,  but in a more restricted setting and with a rather different category of coefficients, which makes the sparsity of the exceptional loci more accessible. Namely, if one assumes $g:\cY\rightarrow S$ factors as $$ g:\cY\stackrel{\iota}{\hookrightarrow}\cX\stackrel{f}{\rightarrow}S$$
 with $\iota:\cY\hookrightarrow \cX$ a closed immersion and $f:\cX\rightarrow S$ an abelian scheme, one can consider the $f$-universally locally acyclic  ($f$-ULA or simply ULA for short) relative perverse sheaf $\cP:=\iota_*\overline{\Q}_\ell[d]$ on $f:\cX\rightarrow S$. As the quotient of the category of $f$-ULA relative perverse sheaves by negligible ones is Tannaka and functorial with respect to pullback along morphism  of schemes $S'\rightarrow S$, one can attach to each $s\in S$  a Tannaka group  $G(\cP)_s$  which detects some of the symmetric features of   $\cY_s$ regarded as a closed subvariety of $ \cX_s$ and ask for the structure of the corresponding degeneracy locus $S_{\cP}\subset S$. \\

\noindent More formally, let $f:\cX\rightarrow S$ be an abelian scheme. Fix a  prime $\ell$.  Let  $D_c^b(\cX)$ denote the  triangulated category    of  \'etale $\overline{\Q}_\ell$-sheaves with bounded constructible cohomology on $\cX$ and let $D^{\ULA}(\cX/S)\subset D_c^b(\cX)$ denote the full subcategory\footnote{The assumption $f$-ULA is not very restrictive as, for every $\cK\in D_c^b(\cX)$ there exists a non-empty open subscheme $U\subset S$ such that $\cK|_{\cX\times_SU}\in D^\ULA(\cX\times_SU/U)$; see \cite[Thm. 2.13, p. 242]{SGA4Half} and \cite[Lemma 3.10]{Barrett}.} of those complexes which are $f$-universally locally acyclic; this is a triangulated subcategory. The  convolution product  built out from the multiplication on $\cX$ endows $D^{\ULA}(\cX/S)$  with a structure of $\overline{\Q}_\ell$-linear rigid symmetric monoidal category.  This monoidal structure, in turn, induces a structure of $\overline{\Q}_\ell$-Tannakian category on the quotient 
  $\Perv^{\ULA}(\cX/S)\twoheadrightarrow P^{\ULA}(\cX/S)$  of the full subcategory   $\Perv^{\ULA}(\cX/S)\subset D^{\ULA} (\cX/S)$ of $f$-ULA relative  perverse $\overline{\Q}_\ell$-sheaf    by the  Serre subcategory of negligible objects. When $S$ is a point, one recovers the usual construction $\Perv(X)\twoheadrightarrow P(X)$ of the $\overline{\Q}_\ell$-Tannakian category of perverse sheaves on an abelian variety $X$. Further, in the relative setting, for every $s\in S$ and geometric point $\overline{s}$ over $s$, the canonical restriction functors 
$$\Perv^\ULA(\cX/S)\stackrel{|_{\cX_s}}{\rightarrow} \Perv(\cX_s)\stackrel{|_{\cX_{\bar s}}}{\rightarrow} \Perv(\cX_{\bar s})$$
induce exact tensor functors
$$P^\ULA(\cX/S)\stackrel{|_{\cX_s}}{\rightarrow}  P(\cX_s)\stackrel{|_{\cX_{\bar s}}}{\rightarrow}  P(\cX_{\bar s}).$$
In particular, fixing a fiber functor $\omega_{\bar s}:P(\cX_{\bar s})\rightarrow \Vect_{ \overline{\Q}_\ell}$, one may ask,  for  $\cP\in \Perv^{\ULA}(\cX/S)$, how the corresponding Tannaka groups $$G(\cP|_{\cX_{\bar s}},\omega_{\bar s}) \subset G(\cP|_{\cX_s},\omega_{\bar s})\subset G(\cP,\omega_{\bar s})$$
vary\footnote{Recall that if $k$ is algebraically closed and $K/k$ is an extension of algebraically closed fields then the canonical restriction functor  $P(X)\rightarrow P(X_K)$, is a fully faithful  tensor functor with image stable under subquotients, so that the category $\langle \cX_{\bar s}\rangle $  and the group $G(\cP|_{\cX_{\bar s}})$ do not depend on the choice of the geometric point $\bar s$ over $s\in S$ but only on $s$ itself. } with $s\in S$. Further, as $S$ is smooth over $k$,  the canonical functor 
$$-|_{\cX_\eta}:\langle \cP\rangle \rightarrow \langle \cP|_{\cX_\eta}\rangle$$ is an equivalence of categories and, for every $s\in S$ one gets a natural (up to inner automorphisms) cospecialization diagram (see Section \ref{Sec:Specialization}):
\begin{equation}\label{Diag:SpecIntro}\xymatrix{1\ar[r]& G(\cP|_{\cX_{\bar s}},\omega_{\bar s})\ar[r]\ar@{_{(}->}[d]&G(\cP|_{\cX_s},\omega_{\bar s})\ar[r]\ar@{_{(}->}[d]&G( \langle\cP|_{\cX_s}\rangle_0,\omega_{\bar s})\ar[r]\ar[d]&1\\
1\ar[r]& G(\cP|_{\cX_{\bar \eta}},\omega_{\bar \eta})\ar[r]&G(\cP|_{\cX_\eta}, \omega_{\bar \eta})\ar[r]&G( \langle\cP|_{\cX_\eta}\rangle_0,\omega_{\bar \eta})\ar[r]&1,\\}\end{equation}
where, for $t\in S$, the category $\langle\cP|_{\cX_t}\rangle_0\subset \langle\cP|_{\cX_t}\rangle$ denotes the full subcategory whose objects are of the form $0_{t*}\cL$ for $\cL\in \Perv(\spec(k(t)))$ and $0_t:\spec(k(t))\rightarrow \cX_t$  the zero-section. As the existence of this cospecialization diagram does not depend on the choice of the fiber functors, we omit them from the notation from now on.  In other words, one would like to understand the arithmetico-geometric structure of the following degeneracy loci:
$$S_{\cP}^?:=\lbrace s\in S \;|\; G(\cP|_{\cX_s} )^?\subsetneq G(\cP|_{\cX_\eta} )^?\rbrace $$
$$S^{\geo,?}_{\cP}:=\lbrace s\in S \;|\; G(\cP|_{\cX_{\bar s}} )^?\subsetneq G(\cP|_{\cX_{\bar \eta}})^?\rbrace, $$
where \textit{e.g.}
$$\begin{tabular}[t]{ll|ll}
 
?$=$&no decoration&$G$&;\\
&$\circ$&$G^\circ$&$:=$ neutral component of $G$;\\
&der&$G^{\der}$&$:=$ derived subgroup of $G$;\\
&$\circ$, der &$G^{\circ, \der}$.& \\
\end{tabular}$$
For $S_{\cP}^{\geo}$, this question has been tackled in  Kr\"{a}mer's dissertation thesis \cite[3.7]{KramerDiss}; in particular Kr\"{a}mer observes that one cannot expect, in general,  that $S_{\cP}^{\geo}$ be a strict, Zariski-closed subset of $S$ unless the determinant  $\det(\cP|_{\cX_{\bar s}})$ of  $\cP|_{\cX_{\bar s}}$ is torsion and uniformly bounded with $s$ \cite[Ex. 3.17 a)]{KramerDiss}. In the converse direction,  Kr\"{a}mer proves the following.

\begin{fait}\label{Fact:KramerDiss}\textit{}\hbox{\rm (\cite[Prop. 3.20)]{KramerDiss}, \cite[Prop. 7.4]{kramer15tannaka})} Let $\cP\in  \Perv^\ULA(\cX/S)$. Assume that for every geometric point $\bar s$ over  $s\in  S$, the restriction $\cP|_{\cX_{\bar s}}$ is simple in  $\Perv(\cX_{\bar s})$ with torsion  determinant\footnote{As a connected reductive group $G$ over an algebraically closed field $Q$ of characteristic $0$ admits an irreducible faithful representation if and only if its center is  $  \G_{m,Q}$ or   finite cyclic, the condition that $\cP|_{\cX_{\bar s}}$ is simple with torsion determinant imposes that $G(\cP|_{\cX_{\bar s}})^\circ$ is semisimple with   finite cyclic center.
}    and that the order of $\det(\cP|_{\cX_{\bar s}})$ is uniformly bounded with $s\in S$. Then   $S^{\geo}_{\cP}$ is not Zariski-dense in $S$.
\end{fait}

\noindent Our main result is about the simplicity assumption.\\

\begin{theoreme}\label{MT}   Let $f:\cX\rightarrow S$ be a morphism, separated and of finite type. For every  $\cP\in \Perv^{\ULA}(\cX/S)$, after possibly replacing $S$ by a non-empty open subscheme (depending on $\cP$) the following holds. For every $s\in S$, 
 $$\hbox{\rm length}_{\Perv(\cX_{\bar \eta})}(\cP|_{\cX_{\bar \eta}})=\hbox{\rm length}_{\Perv(\cX_{\bar s})}(\cP|_{\cX_{\bar s}}).$$
In particular, if $\cP|_{\cX_{\bar \eta}}$ is simple (resp. semisimple)  in $\Perv(\cX_{\bar \eta})$ then $\cP|_{\cX_{\bar s}}$ is simple (resp. semisimple) in $\Perv(\cX_{\bar s})$. 
 \end{theoreme}

\begin{corollaire}\label{Cor:MT} Assume furthermore $f:\cX\rightarrow S$ is an abelian scheme.  Then for every  $\cP\in \Perv^{\ULA}(\cX/S)$, after possibly replacing $S$ by a non-empty open subscheme (depending on $\cP$) the following holds.  For every $s\in S$, 
  $$\hbox{\rm length}_{P(\cX_{\bar \eta})}(\cP|_{\cX_{\bar \eta}})=\hbox{\rm length}_{P(\cX_{\bar s})}(\cP|_{\cX_{\bar s}}).$$
In particular, if $\cP|_{\cX_{\bar \eta}}$ is simple (resp. semisimple)  in $P(\cX_{\bar \eta})$ then $\cP|_{\cX_{\bar s}}$ is simple (resp. semisimple) in $P(\cX_{\bar s})$. 
\end{corollaire}

\noindent  Theorem \ref{MT} yields the following  generalization of  Fact \ref{Fact:KramerDiss} to arbitrary semisimple perverse sheaves. \\

\begin{corollaire}\label{MC2}  Let  $f:\cX\rightarrow S$ be an abelian scheme and let  $\cP\in \Perv^{\ULA}(\cX/S)$. 
\begin{enumerate}[leftmargin=*, parsep=0cm,itemsep=0.2cm,topsep=0.2cm]
\item Assume $\cP|_{\cX_{\bar s}}$ is semisimple in $\Perv(\cX_{\bar s})$ for every $s\in S$. Then   $S_{\cP}^{\geo}$ is   a countable union of strict, Zariski-closed subvarieties of  $S$.
\item Assume $\cP|_{\cX_{\bar \eta }}$ is semisimple in $P(\cX_{\bar \eta})$. Then   $S_{\cP}^{\geo}$ is contained in a countable union of strict, Zariski-closed subvarieties of  $S$.
\end{enumerate}
\end{corollaire}

 \noindent   If $k$ is countable, we do not know if $S_{\cP}^{\geo}\subsetneq S$ in general though we suspect it is true. Still, combined with Fact \ref{Fact:KramerDiss} and some tannakian formalism Theorem \ref{MT} yields the following. For an algebraic group $G$ over a field $Q$, let $R(G)\subset G$ denote its solvable radical (\textit{viz} its largest connected normal solvable subgroup) and  $G\twoheadrightarrow G^{\semis}:=G/R(G)$ its maximal semisimple quotient.  
 
 \begin{corollaire}\label{Cor:MTBis} Let  $f:\cX\rightarrow S$ be an abelian scheme  and let  $\cP\in \Perv^{\ULA}(\cX/S)$. 
 
\begin{enumerate}[leftmargin=*, parsep=0cm,itemsep=0.2cm,topsep=0.2cm]
\item Assume $\cP|_{\cX_{\bar\eta}}$  is simple in $P(\cX_{\bar \eta})$ with torsion determinant. Then  $  S_{\cP}^{\geo}$ 
 is  not Zariski-dense in $S$.
 \item    Up to replacing $S$ by a non-empty open subscheme, one may assume that for   all  $s\in S$  the canonical morphism induced by cospecialization
 $$  G(\cP|_{\cX_{\bar s}} )^\circ \hookrightarrow  G(\cP|_{\cX_{\bar \eta}} )^\circ\twoheadrightarrow G(\cP|_{\cX_{\bar \eta}} )^{\circ,\semis} $$
 factors through an isogeny
  $$\xymatrix{G(\cP|_{\cX_{\bar s}} )^\circ\ar@{^{(}->}[r]\ar@{->>}[d]&G(\cP|_{\cX_{\bar \eta}} )^\circ\ar@{->>}[d]\\
G(\cP|_{\cX_{\bar s}} )^{\circ,\semis}\ar@{.>}[r] &G(\cP|_{\cX_{\bar \eta}} )^{\circ,\semis}.}$$ In particular,
\begin{enumerate}[leftmargin=*, parsep=0cm,itemsep=0.2cm,topsep=0.2cm]
\item   if $G(\cP|_{\cX_{\bar\eta}})$ is  semisimple  (\textit{e.g.} $\cP|_{\cX_{\bar \eta }}$ is simple with torsion determinant in $P(\cX_{\bar \eta})$) then $  S_{\cP}^{\geo,\circ}$  
 is  not Zariski-dense in $S$.

\item if  $G(\cP|_{\cX_{\bar\eta}})$ is reductive (\textit{viz} $\cP|_{\cX_{\bar \eta }}$ is semisimple in $P(\cX_{\bar \eta})$) then  $  S_{\cP}^{\geo,\circ,\der}$
 is  not Zariski-dense in $S$.
\end{enumerate}
\end{enumerate}
\end{corollaire}

\noindent Here is  a sample of geometric application of Corollary \ref{Cor:MTBis} (see also Remark \ref{Rem:GeoProd}). 
 \begin{corollaire}\textit{} \hbox{\rm (Corollary \ref{Cor:GeoProduct})} Let $ \cX\rightarrow S$ be an abelian scheme of relative dimension $g\geq 3$ and $ \cY\hookrightarrow \cX $ a closed subscheme, smooth and geometrically connected  over $S$. Assume $ \cY_{\bar \eta}\hookrightarrow \cX_{\bar \eta}$ has ample normal bundle  and   trivial stabilizer.  Then the set of  all  $s\in S$ such that $\cY_{\bar s}$ is a product is Zariski-dense in $S$ (if and) only if $\cY_{\bar \eta}$ is itself a product.
 \end{corollaire} 
 
 \noindent  We refer to   Section \ref{Sec:GeoApp} for more details. \\

\noindent As for $S_\cP$, at least if $k$ is arithmetically rich enough,  the non-Zariski density  of $S_{\cP}^{\geo}$  in   $S$ automatically implies that $S_\cP$ is sparse in the following  sense.  For an integer $d\geq 1$, write $$|S|^{\leq d}:=\lbrace s\in |S|\;|\; [k(s):k]\leq d\rbrace.$$

\begin{proposition}\label{Prop:Hilbert}   Let  $f:\cX\rightarrow S$ be an abelian scheme  and let  $\cP\in \Perv^{\ULA}(\cX/S)$. Assume    $S_{\cP}^{\geo}$ is  not Zariski-dense in $S$. Assume furthermore that $S$ has dimension $>0$ and that
   $k$ is Hilbertian (\textit{e.g.} finitely generated over $\Q$). Then there exists an integer $d\geq 1$ such that $ |S|^{\leq d}\setminus (S_{\cP}\cap |S|^{\leq d})$ is infinite. 
\end{proposition}

\noindent When $S$ is a curve, $k$ is a number field and  $G(\cP|_{\cX_{\bar k}})$ is  semisimple, the conclusion of Proposition \ref{Prop:Hilbert} can be strengthened to: for every integer $d\geq 1$ the set $  S_{\cP}\cap |S|^{\leq d}$ is finite. This applies, for instance, to the intersection complex $\iota_*\overline{\Q}_\ell[d]$ for $\iota:\cY\hookrightarrow \cX$ a closed immersion such that $\cY\rightarrow S$ is smooth, geometrically connected of relative dimension $d$ and symmetric in the sense that $[-1]^*\cY=\cY$.\\

 \noindent \textbf{Organization of the paper.} In Section \ref{Sec:Review}, we briefly review the Tannakian formalism of perverse sheaves on abelian schemes, both in the absolute and relative setting. In Section \ref{Sec:Specialization}, we elucidate the existence of the specialization diagram (\ref{Diag:SpecIntro}), giving two constructions. The proofs of  Theorem \ref{MT}, its corollaries and Proposition \ref{Prop:Hilbert} are performed in  Section \ref{Sec:Proofs}. The final Section  \ref{Sec:GeoApp} is devoted to a sample of geometric applications of Corollary \ref{Cor:MTBis}.\\
 
 \noindent\textbf{Acknowledgements.}   We thank Emiliano Ambrosi for suggesting the geometric applications analyzed in Section \ref{Sec:GeoApp}, Fran\c{c}ois Charles for pointing out the example in Remark \ref{Rem:GeoProd}, and Lie Fu for asking about 
 an analogue of the Cattani-Deligne-Kaplan theorem, which led to the statement of Corollary \ref{MC2} (1).  We are also grateful to Luc Illusie and Peter Scholze for their answers to our questions about nearby cycles and Theorem \ref{MT}  respectively. We express sincere gratitude to Beat Zurbuchen, for Remark \ref{semistable RPsi} and for pointing out a gap in an earlier version of our proof of  Theorem \ref{MTBis}; the constructive discussions with him helped repair this gap. \\

 \begin{center} \textbf{Notation and conventions}\\  \end{center}

\textit{}\\

\noindent For an additive functor $F:\cA_1\rightarrow \cA_2$ between abelian categories, we write $F:\cA_1\stackrel{\approx}{\rightarrow}\cA_2$ if it is fully faithful with image stable under subquotients and  $F:\cA_1\stackrel{\simeq}{\rightarrow}\cA_2$ if it is  an equivalence.\\

\noindent For a rigid symmetric monoidal category $(\cT,\otimes)$ with unit $\mathbb{I}$ and an object $X$ in $\cT$, let $X^\vee$ denote its dual and, for every integers $m,n\geq 0$, set $T^{m,n}(X):=X^{\otimes m}\otimes X^\vee{}^{\otimes n}$; write $T(X):=\oplus_{m,n\geq 0}T^{m,n}(X)$. If $(\cT,\otimes)$ is Tannakian,  for every integers $m,n\geq 0$, let also $I^{m,n}(X)\subset T^{m,n}(X)$ denote the sum of all subobjects of $T^{m,n}(X)$ which are isomorphic to  $\mathbb{I}$ in $\cT$ (so that $\SHom_{\cT}(\mathbb{I},I^{m,n}(X))\tilde{\rightarrow}\SHom_{\cT}(\mathbb{I},T^{m,n}(X))$). If $\cT$ is Tannakian with fiber functor $\omega:\cT\rightarrow \Vect_Q$, let $G(\cT,\omega)$ denote its Tannaka group; recall that $G(\cT,\omega)$ may depend on $\omega$ but that if $Q$ is algebraically closed then $G(\cT,\omega)$ is uniquely determined up to non-canonical isomorphism. 
For an object $X$ in $\cT$ let $\langle X\rangle\subset \cT$ denote the smallest Tannakian category containing $X$ and, given a fiber functor $\omega:\langle X\rangle\rightarrow \Vect_Q$ set  $G(X,\omega):=G(\langle X\rangle,\omega)$. \\

\noindent In the whole paper, we fix a prime $\ell$. For a scheme $S$,   let  $\Loc(S)$ denote  the category of  \'etale $\overline{\Q}_\ell$-local systems on $S$ and $D_c^b(S)$  the triangulated category of  \'etale $\overline{\Q}_\ell$-sheaves with bounded constructible cohomology on $S$. \\

\noindent A variety  over a field $K$ is a  scheme separated and of finite type over $K$.\\

\noindent When $S$ is a variety, let $\Perv(S)\subset D_c^b(S)$ denote the full subcategory of perverse sheave and, for a morphism $f:X\rightarrow S$   of varieties, write $$D_{X/S}(-):=\hbox{\rm R}\SHom(-,Rf^!\overline{\Q}_\ell): D_c^b(X)^{\op}\rightarrow D_c^b(X)$$ for the relative Verdier duality functor. When $S$ is a point, we simply set $D_X(-):=D_{X/S}(-)$.\\

\noindent For morphisms of  varieties $S_1\rightarrow S\leftarrow S_2$, one writes $$\boxtimes_S^L:D_c^b(S_1)\times D_c^b(S_2)\rightarrow D_c^b(S_1\times_S S_2),\;\; (\cK_1,\cK_2)\mapsto p_1^*\cK_1\otimes^Lp_2^*\cK_2,$$
for the outer tensor product, where $p_i:S_1\times_S S_2\rightarrow S_i$ denotes the $i$th projection, $i=1,2$. When $S=\spec(k)$, one simply writes $\boxtimes^L:=\boxtimes_S^L$.

 \section{Tannakian category of relative perverse sheaves}\label{Sec:Review}
 
 \subsection{Absolute setting}\label{Sec:KWAbsolute} Let $K$ be a field of characteristic $0$ and let   $X$ be an abelian variety over $K$ with group law $m:X\times_KX\rightarrow X$. 
 
 \subsubsection{Construction} See \cite{KramerDiss}, \cite{KW}  for details, and \cite[\S 3.1]{JKLM}  for a shorter overview. The convolution product 
 $$*:D_c^b(X)\times D_c^b(X)\rightarrow D_c^b(X),\;\; (\cK_1,\cK_2)\mapsto \cK_1*\cK_2:=Rm_*( \cK_1\boxtimes^L \cK_2)$$
 endows  $D_c^b(X)$ with the structure of a $\overline{\Q}_\ell$-linear rigid symmetric monoidal category  with duality functor 
 $$(-)^\vee: D_c^b(X)\rightarrow D_c^b(X),\;\; \cK\mapsto \cK^\vee:= [-1]^*D_X(\cK)$$
 and  unit the rank one skyscraper sheaf $\delta_0:=\iota_{0*}\overline{\Q}_\ell \in D_c^b(X)$ supported on $0$.\\ 
 
 \noindent  The full  subcategory  $\Perv(X)\subset D_c^b(X)$    is  abelian  and  stable under  Verdier duality, but not under convolution. To remedy this, one can mod out by  negligible objects. Recall that  every $\cP\in \Perv(X)$ has non negative Euler-Poincar\'e characteristic: $$\chi(X,\cP):=\sum_{i\in \Z}(-1)^i\Sdim_{\overline{\Q}_\ell}(\SH^i(X_{\bar{K}},\cP))\geq 0$$
Let   $^p\SH^n(-):D_c^b(X)\rightarrow \Perv(X)$, $n\in \Z$ denote the perverse cohomology functors and let $N(X)\subset D_c^b(X)$ denote the full subcategory of all $\cK\in   D_c^b(X)$ such that $\chi(X,{}^p\SH^n(\cK))=0$ for all $n\in \Z$; this is a null system such that the convolution bifunctor   $*:D_c^b(X)\times D_c^b(X)\rightarrow D_c^b(X)$, the dualization functor  $(-)^\vee: D_c^b(X)^{\op}\rightarrow D_c^b(X)$  and the perverse cohomology ${}^pH^0(-): D_c^b(X)\rightarrow \Perv(X)$ restrict to 
 $$*:N(X)\times D_c^b(X)\rightarrow N(X),\;\; *:D_c^b(X)\times N(X)\rightarrow N(X)$$
 $$(-)^\vee: N(X)^{\op}\rightarrow N(X)$$
 and 
 $${}^pH^0(-): N(X)\rightarrow N(X)\cap \Perv(X).$$
Consider the quotient functor 
 $$\Perv(X)\rightarrow P(X):=\Perv(X)/(N(X)\cap \Perv(X))$$ 
  so that one gets 
   $$\xymatrix{\Perv(X)\times  \Perv(X) \ar[r]^-{*} \ar[d]& D_c^b(X) \ar[r]^{{}^pH^0(-)}  & \Perv(X)\ar[d] \\
P(X)\times  P(X)\ar@{.>}[rr]_-{*} &&P(X)} $$
 \noindent The  abelian category $P(X)$ endowed with $$*:P(X)\times  P(X)\rightarrow P(X)$$
 is Tannakian with duality functor induced by
    $$\xymatrix{ \Perv(X)^{\op} \ar[r]^{(-)^\vee}\ar[d] & \Perv(X)\ar[d] \\
P(X)^{\op} \ar@{.>}[r]_{(-)^\vee} &P(X)} $$
and unit the image of $\delta_0$ in $P(X)$.

 \subsubsection{Extension of the base field}\label{Sec:BaseFieldExt} See \cite[Sec. 4]{JKLM} for details. Let  $L/K$ be a field extension and let $K^L\subset L$ denote the algebraic closure of $K$ in $L$; assume $K^L/K$ is Galois. Let  $\cT\subset P(X)$ be a full abelian $\otimes$-subcategory and let $\cT_{L}\subset P(X_L)$ denote the  full abelian $\otimes$-subcategory generated by the essential image of $\cT\hookrightarrow P(X)\stackrel{-|_{X_L}}{\rightarrow}P(X_L)$, namely the full subcategory of all $\cQ\in P(X_L)$ such that there exists $\cP\in \cT$ with $\cQ$ a subquotient of $\cP|_{X_L}$. For instance, for every $\cP\in P(X)$, $\langle \cP\rangle_L=\langle \cP|_{X_L}\rangle$. The structure of  $\cT$  is closely related to the structure of $\cT_L$ and the structure of the category $  \Rep_{\overline{\Q}_\ell}(\Gal(K^L/K))$ of finite dimensional continuous $\overline{\Q}_\ell$-representations of the Galois group $\Gal(K^L/K)$ of $K^L/K$. More precisely,
   \begin{itemize}[leftmargin=*,label=-]
   \item  The canonical functor $$-|_{X_{L}}:\Perv(X)\stackrel{|_{X_L}}{\rightarrow} \Perv(X_L)\rightarrow P(X_L)$$ is an exact   functor of $\overline{\Q}_\ell$-linear  categories which induces a faithful functor of Tannakian categories
   $$\xymatrix{\Perv(X)\ar[r]^{|_{X_L}} \ar[d]& \Perv(X_L)\ar[d]\\
  P(X)\ar@{.>}[r]^{|_{X_L}}  & P(X_L)}.$$
    
\item  For simplicity, write $\Perv(K):=\Perv(\spec(K))$. The canonical functor $$0_*: \Perv(K)\stackrel{0_*}{\rightarrow} \Perv(X)\rightarrow P(X)$$ is an exact   fully faithful functor of Tannakian categories with essential image $P_0(X)\subset P(X)$  stable under subquotients. Precomposing $0_*: \Perv(K) \rightarrow P_0(X)\hookrightarrow P(X)$ with the fully faithful exact $\otimes$-tensor functor $  \Rep_{\overline{\Q}_\ell}(\Gal(K^L/K))\hookrightarrow \Perv(K)$, one gets an exact   fully faithful functor of Tannakian categories
 $$0_*^L: \Rep_{\overline{\Q}_\ell}(\Gal(K^L/K))\rightarrow P(X);$$
let $P_0^L(X)\subset P_0(X)$ denote its essential image.\\

 \end{itemize}  
 
 \noindent Consider the full Tannakian subcategory $\cT_0^L:=\cT\cap P_0^L(X) \subset \cT$. Then, for every fiber functor $\omega: \cT_L\rightarrow \Vect_{\Q_\ell}$, the sequence of Tannakian categories $$ \cT_0^L \rightarrow \cT\stackrel{|_{X_L}}{\rightarrow} \cT_L$$
induces a short exact sequence of proalgebraic groups 
$$\xymatrix{1\ar[r]&G( \cT_L,\omega)\ar[r]&  G( \cT,\omega)\ar[r]& G( \cT_0^L ,\omega)\ar[r]& 1,}$$
from which one immediately deduces   that

  \begin{itemize}[leftmargin=*, parsep=0cm,itemsep=0.2cm,topsep=0.2cm,label=-]
   \item  (\ref{Sec:BaseFieldExt}-1) For every $\cP\in P(X)$,  the sequence of Tannakian categories $$\langle \cP\rangle_0^L \rightarrow  \langle \cP\rangle \stackrel{|_{X_L}}{\rightarrow} \langle \cP|_{X_L}\rangle$$
induces a short exact sequence of  algebraic groups 
$$\xymatrix{1\ar[r]&G(\cP|_{X_L},\omega)\ar[r]& G(\cP ,\omega)\ar[r]& G(\langle \cP\rangle_0^L,\omega)\ar[r]& 1.}$$
 
 \item   (\ref{Sec:BaseFieldExt}-2) If $K$ is algebraically closed, the  restriction functor $-|_L: \cT\stackrel{\simeq}{\rightarrow} \cT_L$ is an equivalence of Tannakian categories. In particular, for every $\cP\in P(X)$, $G(\cP|_{X_L},\omega)\tilde{\rightarrow} G(\cP ,\omega)$.
 \end{itemize} 
 
 \noindent We drop the superscript $(-)^L$ when $L=\overline{K}$.

  \subsection{Relative setting}\label{Sec:KWRel}Let $k$ be a  field of characteristic $0$, $S$  a smooth,  geometrically connected  variety over $k$  with generic point $\eta$.   Let   $f:\cX\rightarrow S$ be an abelian scheme.   \\
   \subsubsection{Construction} See  \cite{HansenScholze} for details. Let $D^{\ULA}(\cX/S)\subset D_c^b(\cX)$ denote the full subcategory of    $f$-universally locally acyclic ($f$-ULA or   just ULA for short) complexes on $\cX/S$; this is a triangulated subcategory.
   As in the absolute setting, the convolution product 
    $$*:D^\ULA(\cX/S) \times D^\ULA(\cX/S)\rightarrow D^\ULA(\cX/S),\;\; (\cK_1,\cK_2)\mapsto \cK_1*\cK_2:=Rm_*( \cK_1\boxtimes_S^L \cK_2)$$
 endows  $D^\ULA(\cX/S)$ with the structure of a $\overline{\Q}_\ell$-linear rigid symmetric monoidal category  with duality functor 
 $$(-)^\vee: D^\ULA(\cX/S)\rightarrow D^\ULA(\cX/S),\;\; K\mapsto \cK^\vee:= [-1]^*D_{\cX/S}(\cK)$$
and  unit  $\delta_{S,0}:= 0_*\overline{\Q}_\ell \in D_c^b(\cX)$, where $ 0:S\hookrightarrow \cX$ is the $0$-section. By construction and proper base change, for every $s\in S$, the pull-back functor $-|_{\cX_s}:D^\ULA(\cX/S)\rightarrow D_c^b(\cX_s)$  is a tensor functor.  \\

 \noindent     Let   $\Perv^\ULA(\cX/S)\subset D^{\ULA}(\cX/S)$ denote the full subcategory of  ULA relative perverse sheaves on $\cX$. This is an abelian category, stable by relative Verdier duality $D_{\cX/S}(-):D_c^b(\cX)^{\op}\rightarrow D_c^b(\cX)$
 and  such that for every $s\in S$, the pull-back functor $-|_{\cX_s}:D_c^b(\cX)\rightarrow D_c^b(\cX_s)$ restricts to an exact functor $-|_{\cX_s}:\Perv^\ULA(\cX/S)\rightarrow \Perv(\cX_s)$ which, when $s=\eta$, is   fully faithful with essential image stable under subquotients. Actually,  $\Perv^\ULA(\cX/S)\subset D^{\ULA}(\cX/S)$ is the heart of a $t$-structure $D^{\ULA,\leq 0}(\cX/S),D^{\ULA,\geq 0}(\cX/S)\subset D^{\ULA}(\cX/S)$  - the relative perverse $t$-structure with associated truncation functors ${}^{p/S}\tau^{\leq 0}:D^{\ULA}(\cX/S)\rightarrow D^{\ULA,\leq 0}(\cX/S)$, ${}^{p/S}\tau^{\geq 0}:D^{\ULA}(\cX/S)\rightarrow D^{\ULA,\geq 0}(\cX/S)$ and perverse cohomology functors 
 $${}^{p/S}\SH^n:D^{\ULA}(\cX/S)\rightarrow \Perv^{\ULA}(\cX/S),\;\; n\in \Z.$$ 
 
\noindent As $f:\cX\rightarrow S$  is proper, $Rf_*:D^\ULA(\cX/S)\rightarrow D_c^b(S)$ factors as $Rf_*:D^\ULA(\cX/S)\rightarrow D^\ULA(S/S)= D_{\liss}^b(S)$ \cite[Lem. (ii), p.20, Lem. (i), p.21]{Barrett}.  Combined with the fact that, for every $t\in S$, the following diagrams
$$\xymatrix{D^\ULA(\cX/S)\ar[r]^{-|_{\cX_{\bar t}}}\ar[d]_{^{p/S}\hbox{\rm \tiny H}^n(-)}&D_c^b(\cX_{\bar t})\ar[d]^{^p\hbox{\rm \tiny H}^n(-)}\\
\Perv^\ULA(\cX/S)\ar[r]_{-|_{\cX_{\bar t}}}&\Perv(\cX_{\bar t})},\;\; n\in \Z$$
commute, one gets that, for  $\cK\in  D^\ULA(\cX/S)$, the following properties are equivalent:
 
 \begin{enumerate}[label=(\roman*)]
\item  $\cK|_{\cX_\eta}\in N(\cX_\eta)$;
 \item For every $s\in S$, $\cK|_{\cX_s}\in N(\cX_s)$;
  \item There exists $s\in S$ such that $\cK|_{\cX_s}\in N(\cX_s)$.
 \end{enumerate}  
\noindent In other words, for every $s\in S$,  the following null systems of $D^{\ULA}(\cX/S)$
\begin{equation}\label{Diag:DefN}\begin{tabular}[t]{ll}
$N^{\ULA}(\cX/S)$&$:=\ker(D^{\ULA}(\cX/S)\stackrel{-|_{\cX_\eta}}{\rightarrow}D_c^b(\cX_\eta)\rightarrow D_c^b(\cX_\eta)/N(\cX_\eta))$\\
&$=\ker(D^{\ULA}(\cX/S)\stackrel{-|_{\cX_s}}{\rightarrow}D_c^b(\cX_s)\rightarrow D_c^b(\cX_s)/N(\cX_s))$.
\end{tabular}\end{equation}
coincide.\\

 \noindent By  construction  the functors 
   $*:D^{\ULA}(\cX/S)\times D^{\ULA}(\cX/S)\rightarrow D^{\ULA}(\cX/S)$,  $(-)^\vee: D^{\ULA}(\cX/S)^{\op}\rightarrow D^{\ULA}(\cX/S)$, ${}^{p/S}H^0(-): D^{\ULA}(\cX/S)\rightarrow \Perv^{\ULA}(\cX/S)$ and $-|_{\cX_s}: D^{\ULA}(\cX/S)\rightarrow D_c^b(\cX_s)$, $s\in S$ restrict to 
 $$*:N^{\ULA}(\cX/S)\times D^{\ULA}(\cX/S)\rightarrow N^{\ULA}(\cX/S),\;\; *:D^{\ULA}(\cX/S)\times N^{\ULA}(\cX/S)\rightarrow N^{\ULA}(\cX/S)$$
 $$(-)^\vee:N^{\ULA}(\cX/S)^{\op}\rightarrow N^{\ULA}(\cX/S)$$
 $${}^{p/S}H^0(-): N^{\ULA}(\cX/S)\rightarrow N^{\ULA}(\cX/S)\cap \Perv^{\ULA}(\cX/S)$$
 and 
 $$-|_{\cX_s}: N^{\ULA}(\cX/S)\rightarrow N(\cX_s), \;\; s\in S.$$
Consider the quotient functor 
 $$ \Perv^{\ULA}(\cX/S)\rightarrow P^{\ULA}(\cX/S):=\Perv^{\ULA}(\cX/S)/(\Perv^{\ULA}(\cX/S)\cap N^{\ULA}(\cX/S))$$ 
  so that one gets 
$$\xymatrix{  \Perv^{\ULA}(\cX/S) \times   \Perv^{\ULA}(\cX/S) \ar[r]^-{*} \ar[d]& D^{\ULA}(\cX/S) \ar[r]^{{}^{p/S}H^0(-)} \ar[r] & \Perv^{\ULA}(\cX/S)\ar[d] \\
P^{\ULA}(\cX/S)\times  P^{\ULA}(\cX/S)\ar@{.>}[rr]_-{*} &&P^{\ULA}(\cX/S)} $$
 \noindent The  abelian category $P^{\ULA}(\cX/S)$ endowed with $$*:P^{\ULA}(\cX/S)\times  P^{\ULA}(\cX/S)\rightarrow P^{\ULA}(\cX/S)$$
 is a $\overline{\Q}_\ell$-linear rigid symmetric monoidal category with duality functor induced by
    $$\xymatrix{ \Perv^{\ULA}(\cX/S)^{\op} \ar[r]^{(-)^\vee}\ar[d] & \Perv^{\ULA}(\cX/S)\ar[d] \\
P^{\ULA}(\cX/S)^{\op}  \ar@{.>}[r]_{(-)^\vee} &P^{\ULA}(\cX/S)} $$
and unit the image of $\delta_{S,0} $ in $P^\ULA(\cX/S)$. For every $s\in S$, the exact pull-back functor $-|_{X_s}:\Perv^{\ULA}(\cX/S)\rightarrow \Perv (\cX_s)$ induces  an    exact  faithful functor of $\overline{\Q}_\ell$-linear rigid symmetric monoidal categories
 \begin{equation}\label{Eq:1}\xymatrix{ \Perv^{\ULA}(\cX/S)  \ar[r]^{-|_{\cX_s}}\ar[d] & \Perv(\cX_s)\ar[d] \\
P^{\ULA}(\cX/S) \ar@{.>}[r]_{ -|_{\cX_s}} &P(\cX_s),} \end{equation}
which, when $s=\eta$, is  fully faithful with essential image stable under subquotients. In particular, $-|_{\cX_\eta}:P^{\ULA}(\cX/S)\rightarrow P(\cX_\eta)$ identifies $P^{\ULA}(\cX/S)$ with a full Tannakian subcategory of $ P(\cX_\eta)$ and for every $\cP\in P^{\ULA}(\cX/S) $, induces an equivalence of Tannakian categories
$$-|_{\cX_\eta}:\langle \cP\rangle \stackrel{\simeq}{\rightarrow} \langle \cP|_{\cX_\eta}\rangle.$$

 \section{Specialization}\label{Sec:Specialization} 
\noindent   In the following, to simplify notation, given an exact tensor functor of Tannakian categories $F:\cT'\rightarrow \cT$ and a fiber functor $\omega$ on $\cT$, we will again write $\omega:=\omega\circ F$ for the resulting fiber functor on $\cT'$; as the functor $F:\cT'\rightarrow \cT$ should always be clear from the context, this should not give rise to confusion.\\
 
\noindent  Let $k$ be a  field of characteristic $0$, let $S$  a smooth,  geometrically connected  variety over $k$  with generic point $\eta$.   Let   $f:\cX\rightarrow S$ be an abelian scheme. From (\ref{Diag:DefN}), the  canonical diagram of $\overline{\Q}_\ell$-linear abelian categories  (\ref{Diag:Tannaka1}-1) induces the canonical diagram of Tannakian categories  (\ref{Diag:Tannaka1}-2)

 \begin{equation}\label{Diag:Tannaka1}
 	\resizebox{1.2\textwidth}{!}{%
 		\begin{tabular}[t]{llll}
 (\ref{Diag:Tannaka1}-1)&
 $$\xymatrix{  \Perv(\cX_s)_0\ar[r]&\Perv(\cX_s)\ar[r]^{|_{\cX_{\bar s}}}&\Perv(\cX_{\bar s})\\
 &\Perv^{\ULA}(\cX/S)\ar[u]^{|_{\cX_s}}\ar[d]_{|_{\cX_\eta}}^\approx &\\
 \Perv(\cX_\eta)_0\ar[r]& \Perv(\cX_\eta) \ar[r]^{|_{\cX_{\bar \eta}}}& \Perv(\cX_{\bar \eta})\\ }$$& (\ref{Diag:Tannaka1}-2)& $$\xymatrix{ P(\cX_s)_0\ar[r]&P(\cX_s)\ar[r]^{|_{\cX_{\bar s}}}&P(\cX_{\bar s})\\
 &P^{\ULA}(\cX/S)\ar[u]^{|_{\cX_s}}\ar[d]_{|_{\cX_\eta}}^\approx &\\
P(\cX_\eta)_0\ar[r]&P(\cX_\eta) \ar[r]^{|_{\cX_{\bar \eta}}}& P(\cX_{\bar \eta})\\ }$$
\end{tabular}%
}
\end{equation}
and  for every    $\cP\in P^\ULA(\cX/S)$,  
\begin{equation}\label{Diag:Tannaka}\xymatrix{\langle \cP|_{\cX_s}\rangle_0\ar[r]&\langle \cP|_{\cX_s}\rangle\ar[r]^{|_{\cX_{\bar s}}}&\langle \cP|_{\cX_{\bar s}}\rangle\\
 &\langle \cP\rangle\ar[u]^{|_{\cX_s}}\ar[d]_{|_{\cX_\eta}}^{\simeq }&\\
\langle \cP|_{\cX_\eta}\rangle_0\ar[r]&\langle \cP|_{\cX_{\eta }}\rangle \ar[r]^{|_{\cX_{\bar \eta}}}&\langle \cP|_{\cX_{\bar \eta}}\rangle.\\ }\end{equation}
For every $t\in S$, fix  a fiber functor $\omega_{\bar t}:\langle \cP|_{\cX_{\bar t}}\rangle\rightarrow \Vect_{\overline{\Q}_\ell}$.  
   We claim that   for every $s\in S$ and choice of an isomorphism of fiber functors $$\omega_{\bar s}\circ -|_{\cX_{\bar s}}\tilde{\rightarrow} \omega_{\bar \eta}\circ -|_{\cX_{\bar \eta}}: P^\ULA(\cX/S)\rightarrow \Vect_{\overline{\Q}_\ell},$$ 
the diagram (\ref{Diag:Tannaka})  induces a diagram of algebraic groups with exact lines, which can be completed as indicated by the dotted arrows
\begin{equation}\label{Diag:Rep}\xymatrix{1\ar[r]& G(\cP|_{\cX_{\bar s}},\omega_{\bar s})\ar[r]\ar@{_{(}.>}[ddd]^{ csp_{\bar \eta,\bar s}}&G(\cP|_{\cX_s},\omega_{\bar s})\ar[r]\ar@{_{(}->}[d]&G( \langle\cP|_{\cX_s}\rangle_0,\omega_{\bar s})\ar[r]\ar@{.>}[ddd]^{  (csp_{\bar \eta,\bar s})_0}&1\\
 && G(\cP,\omega_{\bar s})\ar[d]^\simeq&&\\
 && G(\cP,\omega_{\bar \eta})&&\\
1\ar[r]& G(\cP|_{\cX_{\bar \eta}},\omega_{\bar \eta})\ar[r]&G(\cP|_{\cX_\eta}, \omega_{\bar \eta})\ar[r]\ar[u]_\simeq&G( \langle\cP|_{\cX_\eta}\rangle_0,\omega_{\bar \eta})\ar[r]&1\\}\end{equation}
The exactness of the lines is (\ref{Sec:BaseFieldExt}-1). The fact that $  G(\cP|_{\cX_s} ,\omega_{\bar s})\hookrightarrow G(\cP,\omega_{\bar s})$ is a  closed immersion is formal\footnote{Indeed, by definition of $\langle \cP|_{\cX_s}\rangle$, every object  in $\langle \cP|_{\cX_s}\rangle$ is a subquotient of some $T^{m,n}(\cP|_{\cX_s})\simeq T^{m,n}(\cP)|_{\cX_s}$ for some integers $m,n\geq 0$.}. \\

\noindent Note that the existence of the dotted arrows is independent of the choice of the isomorphism  $\omega_{\bar s}\circ -|_{\cX_{\bar s}}\tilde{\rightarrow} \omega_{\bar \eta}\circ -|_{\cX_{\bar \eta}}$  hence of the fiber functors $\omega_{\bar \eta}$, $\omega_{\bar s}$. So we will be free to choose  $\omega_{\bar \eta}$, $\omega_{\bar s}$.   Note also that the existence of the arrow  $csp_{\bar \eta,\bar s}$ is equivalent to the one of the arrow  $(csp_{\bar \eta,\bar s})_0$. We provide two constructions of (\ref{Diag:Rep}), one  \textit{via} a construction of  $csp_{\bar \eta,\bar s}$ and one \textit{via} a construction of  $(csp_{\bar \eta,\bar s})_0$. In both cases, we actually complete  (\ref{Diag:Tannaka1}-1),  (\ref{Diag:Tannaka1}-2) (\ref{Diag:Tannaka}) by introducing some intermediate categories (to be defined) - $(*)_{\bar \eta,\bar s}$ to construct $csp_{\bar \eta,\bar s}$ and $((*)_{\bar \eta,\bar s})_0$ to construct $(csp_{\bar \eta,\bar s})_0$ as indicated in  (\ref{Diag:Tannaka1Comp}-1),  (\ref{Diag:Tannaka1Comp}-2) and (\ref{Diag:TannakaComp}) below.

 \begin{equation}\label{Diag:Tannaka1Comp}
 		\resizebox{1.2\textwidth}{!}{%
 	\begin{tabular}[t]{llll}
 (\ref{Diag:Tannaka1Comp}-1)&
 $$\xymatrix{  \Perv(\cX_s)_0\ar[r]&\Perv(\cX_s)\ar[r]^{|_{\cX_{\bar s}}}&\Perv(\cX_{\bar s})\\
 ((*)_{\bar \eta,\bar s})_0\ar@{.>}[u]\ar@{.>}[d]^\simeq\ar@{.>}[r]  &\Perv^{\ULA}(\cX/S)\ar[u]^{|_{\cX_s}}\ar[d]_{|_{\cX_\eta}}^\approx\ar@{.>}[r] &(*)_{\bar \eta,\bar s}\ar@{.>}[u]\ar@{.>}[d]^\simeq\\
 \Perv(\cX_\eta)_0\ar[r]& \Perv(\cX_\eta) \ar[r]^{|_{\cX_{\bar \eta}}}& \Perv(\cX_{\bar \eta})\\ }$$& (\ref{Diag:Tannaka1Comp}-2)& $$\xymatrix{ P(\cX_s)_0\ar[r]&P(\cX_s)\ar[r]^{|_{\cX_{\bar s}}}&P(\cX_{\bar s})\\
  ((*)_{\bar \eta,\bar s})_0\ar@{.>}[u]\ar@{.>}[d]^\simeq\ar@{.>}[r] &P^{\ULA}(\cX/S)\ar[u]^{|_{\cX_s}}\ar[d]_{|_{\cX_\eta}}^\approx\ar@{.>}[r] &  (*)_{\bar \eta,\bar s}\ar@{.>}[u]\ar@{.>}[d]^\simeq\\
P(\cX_\eta)_0\ar[r]&P(\cX_\eta) \ar[r]^{|_{\cX_{\bar \eta}}}& P(\cX_{\bar \eta})\\ }$$
\end{tabular}%
}\end{equation}

\begin{equation}\label{Diag:TannakaComp} \xymatrix{\langle \cP|_{\cX_s}\rangle_0\ar[r]&\langle \cP|_{\cX_s}\rangle\ar[r]^{|_{\cX_{\bar s}}}&\langle \cP|_{\cX_{\bar s}}\rangle\\
((*)_{\bar \eta,\bar s})_0\ar@{.>}[u]\ar@{.>}[d]^\simeq\ar@{.>}[r] &\langle \cP\rangle\ar[u]^{|_{\cX_s}}\ar[d]_{|_{\cX_\eta}}^\simeq\ar@{.>}[r] &(*)_{\bar \eta,\bar s}\ar@{.>}[u]\ar@{.>}[d]^\simeq\\
\langle \cP|_{\cX_\eta}\rangle_0\ar[r]&\langle \cP|_{\cX_{\eta }}\rangle \ar[r]^{|_{\cX_{\bar \eta}}}&\langle \cP|_{\cX_{\bar \eta}}\rangle\\ }\end{equation}

\begin{remarque}\label{Rem:Spec}\textnormal{Actually, the constructions of  (\ref{Diag:Tannaka1Comp}-1)  do not use that $f:\cX\rightarrow S$ is an abelian scheme; they only require  that  $f:\cX\rightarrow S$ be separated and of finite type (with a section). }
\end{remarque}
 \subsection{A construction of $(csp_{\bar \eta,\bar s})_0$}\label{Sec:Construction1}
 
\noindent We begin with the following observation. 

 \begin{lemme}\label{Lem:SpecAr}  Let    $f:\cX\rightarrow S$ be a separated morphism of finite type  with a section $\iota:S\rightarrow \cX$. The following commutative diagram 
$$\xymatrix{\Loc(S)=\Perv^{\ULA}(S/S)\ar[r]^(0.6){\iota_*}\ar[d]_{\eta^*}&\Perv^{\ULA}(\cX/S)\ar[d]^{|_{\cX_\eta}}\\
\Perv(\eta)\ar[r]^{\iota_{\eta *}}&\Perv(\cX_\eta)}$$
is cartesian. Namely for every $\cL_{[\eta]}\in \Perv(\eta)$, if there exists $\cP\in \Perv^\ULA(\cX/S)$ such that $\iota_{\eta *}\cL_{[\eta]}\simeq \cP|_{\cX_\eta}$ then there exists $\cL\in \Loc(S)$ such that $\eta^*\cL\simeq \cL_{[\eta]}$ and $\iota_*\cL\simeq\cP$.
 \end{lemme}
 
 \begin{proof} As $f:\cX\rightarrow S$ is separated, $\iota: S\hookrightarrow \cX$ is a closed immersion; let $j:U:=\cX\setminus \iota(S)\hookrightarrow \cX$ denote the complementary open immersion. Then $(j^*\cP)|_{U_\eta}\simeq j_\eta^*(\cP|_{\cX_\eta})\simeq  j_\eta^*\iota_{\eta *}\cL_{[\eta]}\simeq 0$. As $j^*\cP\in \Perv^{\ULA}(U/S)$ and $-|_{U_\eta}:\Perv^{\ULA}(U/S)\rightarrow \Perv(U_\eta)$ is fully faithful, this forces $j^*\cP=0$. From the distinguished triangle 
 $$j_!j^*\cP\rightarrow \cP\rightarrow \iota_*\iota^*\cP\stackrel{+1}{\rightarrow}$$
in $D_c^b(\cX)$,  $\cP\tilde{\rightarrow} \iota_*\iota^*\cP$ hence, by  \cite[Lem. 3.6 (iv)]{Barrett},  $\iota^*\cP\in D^\ULA(S/S)$. From \cite[Lem. 3.7 (ii)]{Barrett}, $\iota^*\cP\in D_{\liss}^b(S)$. But $\eta^*\iota^*\cP\simeq \iota_\eta^*(\cP|_{\cX_\eta})\simeq \iota_\eta^*\iota_{\eta *}\cL_{[\eta]}\simeq \cL_{[\eta]}$, so that $\cL:=\iota^*\cP$ lies in $\Loc(S)$ and has the requested property.
 \end{proof}
 
 \noindent We return to the case where  $f:\cX\rightarrow S$ is an abelian scheme and $\cP\in \Perv^{\ULA}(\cX/S)$. From Lemma \ref{Lem:SpecAr}, one can complete (\ref{Diag:Tannaka}) as 
\begin{equation}\label{Diag:TannakaComp1}\xymatrix{\langle \cP|_{\cX_s}\rangle_0\ar[r]&\langle \cP|_{\cX_s}\rangle\ar[r]^{|_{\cX_{\bar s}}}&\langle \cP|_{\cX_{\bar s}}\rangle\\
\langle \cP \rangle_0\ar[r]\ar[d]_{|_{\cX_\eta}}^\simeq \ar[u]^{|_{\cX_s}}&\langle \cP\rangle\ar[u]^{|_{\cX_s}}\ar[d]_{|_{\cX_\eta}}^\simeq&\\
\langle \cP|_{\cX_\eta}\rangle_0\ar[r]&\langle \cP|_{\cX_{\eta }}\rangle \ar[r]^{|_{\cX_{\bar \eta}}}&\langle \cP|_{\cX_{\bar \eta}}\rangle,\\ }\end{equation}
 which, as claimed, formally yields a commutative diagram of algebraic groups:
 
 \begin{equation}\label{Diag:RepComp1}\xymatrix{1\ar[r]& G(\cP|_{\cX_{\bar s}},\omega_{\bar s})\ar[r]\ar@{_{(}.>}[ddd]^{csp_{\bar \eta,\bar s} }&G(\cP|_{\cX_s},\omega_{\bar s})\ar[r]\ar@{_{(}->}[d]&G( \langle\cP|_{\cX_s}\rangle_0,\omega_{\bar s})\ar[r]\ar[d]\ar@/^3pc/[ddd]^{(csp_{\bar \eta,\bar s})_0}&1\\
 && G(\cP,\omega_{\bar s})\ar[d]^\simeq\ar[r]&G( \langle\cP \rangle_0,\omega_{\bar s})\ar[r]\ar[d]^\simeq&1\\
 && G(\cP,\omega_{\bar \eta})\ar[r]&G( \langle\cP \rangle_0,\omega_{\bar \eta})\ar[r]&1\\
1\ar[r]& G(\cP|_{\cX_{\bar \eta}},\omega_{\bar \eta})\ar[r]&G(\cP|_{\cX_\eta}, \omega_{\bar \eta})\ar[r]\ar[u]_\simeq&G( \langle\cP|_{\cX_\eta}\rangle_0,\omega_{\bar \eta})\ar[r]\ar[u]_\simeq&1.\\}\end{equation}

 \subsection{A construction of $csp_{\bar \eta,\bar s} $}\label{Sec:Construction2} 

\subsubsection{Absolutely integrally closed valuation rings and nearby cycles}Recall that a  valuation ring $V$ is said to be absolutely integrally closed (AIC for short)  if it satisfies the following equivalent conditions 
 
\begin{enumerate}[leftmargin=3cm, parsep=0cm,itemsep=0.2cm,topsep=0.2cm,label=(AIC-\arabic*)]
\item The fraction field of $V$ is algebraically closed;
\item Every monic polynomial of degree $\geq 1$ in $V[T]$ has a root in $V$.
\end{enumerate}
In  particular, such a valuation ring  $V$ is  strictly henselian.   \\

\begin{fait}\label{HS}\textit{}\textnormal{(\cite[Thm. 1.7, Thm. 6.1 (ii), Cor. 4.2]{HansenScholze})} Let $S=\spec(V)$ be the spectrum of an AIC valuation ring with generic point $\eta$ and closed point $s$.  Let $f:\cX\rightarrow S$ be a morphism, separated and of finite presentation and write $\cX_\eta\stackrel{\beta}{\hookrightarrow}\cX\stackrel{\alpha}{\hookleftarrow}\cX_s$ for the inclusions of the generic and closed fibers respectively. Then,\begin{enumerate}[leftmargin=*]
\item The functor $\beta^*:D^{\ULA}(\cX/S )\rightarrow D_c^b(\cX_\eta )$ is an equivalence of   categories with quasi-inverse $R\beta_*:D_c^b(\cX_\eta )\rightarrow D^{\ULA}(\cX/S )$. In particular,  $\beta^*:D^{\ULA}(\cX/S)\rightarrow D_c^b(\cX_\eta )$ restricts to   an equivalence of   categories $\beta^*:\Perv^{\ULA}(\cX/S )\rightarrow \Perv(\cX_\eta )$ with quasi-inverse $R\beta_*:\Perv(\cX_\eta )\rightarrow \Perv^{\ULA}(\cX/S )$. 
\item The nearby cycle functor $R\psi_f=\alpha^*R\beta_*:D(\cX_\eta )\rightarrow D(\cX_s )$ restricts to a functor $R\psi_f :D_c^b(\cX_\eta )\rightarrow D_c^b(\cX_s )$ which is t-exact with respect to the perverse t-structures hence induces an exact functor  $R\psi_f :\Perv(\cX_\eta )\rightarrow \Perv(\cX_s )$.
\item Assume furthermore that $f:\cX\rightarrow S$  is an abelian scheme. Then $R\psi_f :D_c^b(\cX_\eta )\rightarrow  D_c^b(\cX_s )$ is a tensor functor and  $N(\cX_{\eta})\ =\ker( D_c^b(\cX_\eta)\stackrel{R\psi_f}{\rightarrow} D_c^b(\cX_s )\rightarrow D_c^b(\cX_s)/N(\cX_s)$; in particular,  $R\psi_f :\Perv(\cX_\eta )\rightarrow \Perv(\cX_s )$ induces a faithful exact tensor functor $R\psi_f :P(\cX_\eta )\rightarrow  P(\cX_s )$.
\end{enumerate}
\end{fait}

\noindent From   \cite[Lem. 3.28]{BhattMathew}, for a quasi-compact, quasi-separated scheme $T$, a specialization $t_1\rightsquigarrow t$ of points on $T$, one can always find  a morphism $S\rightarrow T$ with  source the spectrum $S=\Spec(V)$ of  an AIC valuation ring $V$,  mapping the generic point $\eta$  (resp. the closed point $s$) of $S$ to $t_1$ (resp. $t$).  We will call such a morphism - usually written  as $(S,\eta,s)\rightarrow (T,t_1,t)$, a \textit{witness} for  $t_1\rightsquigarrow t$ in $T$.  The proof of  \cite[Lem. 3.28]{BhattMathew} shows that, if furthermore one fixes a geometric point    $\overline{t}_1$ over $t_1$, one can choose $S$ in such a way that $\eta\rightarrow t_1$ factors as $\eta\rightarrow \overline{t}_1\rightarrow t_1$; if we want to specify a geometric point over which $\eta\rightarrow t_1$ factors, we will rather write $(S,\eta,s)\rightarrow (T,\overline{t}_1,t)$.

\subsubsection{}\label{Sec:Cons}We return to the case where $S$ is a smooth, geometrically connected variety over $k$ and $f:\cX\rightarrow S$ is an abelian scheme. For every specialization  $  \eta \rightsquigarrow s$ of   points  on $S$, fix a witness  $(S',\eta',s')\rightarrow (S,\eta, s)$ and geometric points $\eta'\rightarrow \bar \eta\rightarrow \eta$, $s'\rightarrow \bar s\rightarrow s$. Set $f':\cX':=\cX\times_SS'\rightarrow S'$. From Fact \ref{HS}, one gets a canonical diagram of  $\overline{\Q}_\ell$-linear abelian categories 
$$\xymatrix{&\Perv(\cX_s)\ar[r]^{|_{\cX_{\bar s}}}&\Perv(\cX_{\bar s})\ar[r]^{|_{\cX'_{ s'}}}&\Perv(\cX'_{ s'})\\
\Perv^\ULA(\cX/S)\ar[r]^{|_{\cX'}}\ar[ur]^{|_{\cX_s}}\ar[dr]_{|_{\cX_\eta}}^\approx&\Perv^\ULA(\cX'/S')\ar@/_0.5cm/[urr]^{|_{\cX'_{s'}}}\ar@/^0.5cm/[drr]_{|_{\cX'_{\eta'}}}^\simeq&&\\
&\Perv(\cX_\eta)\ar[r]_{|_{\cX_{\bar \eta}}}&\Perv(\cX_{\bar \eta})\ar[r]_{|_{\cX'_{  \eta'}}}&\Perv(\cX'_{ \eta'})\ar@{.>}@/_2pc/[uu]_{R\psi_{f'}}}$$
which induces, for every   $\cP\in \Perv^{\ULA}(\cX/S)$, a canonical diagram of Tannakian categories
\begin{equation}\label{Diag:TannakaComp2}\xymatrix{&\langle \cP|_{\cX_s}\rangle \ar[r]^{|_{\cX_{\bar s}}}&\langle \cP|_{\cX_{\bar s}}\rangle\ar[r]^{|_{\cX'_{ s'}}}_\simeq &\langle \cP|_{\cX'_{ s'}}\rangle\\
 \langle \cP \rangle\ar[r]^{|_{\cX'}}\ar[ur]^{|_{\cX_s}}\ar[dr]_{|_{\cX_\eta}}^\simeq& \langle \cP|_{\cX'} \rangle\ar@/_0.5cm/[urr]^{|_{\cX'_{s'}}}\ar@/^0.5cm/[drr]^\simeq_{|_{\cX'_{\eta'}}}&&\\
&\langle \cP|_{\cX_\eta}\rangle\ar[r]_{|_{\cX_{\bar \eta}}} &\langle \cP|_{\cX_{\bar \eta}}\rangle\ar[r]^\simeq_{|_{\cX'_{\eta'}}} &\langle \cP|_{\cX'_{ \eta'}}\rangle\ar@{.>}@/_2pc/[uu]_{R\psi_{f'}}},\end{equation}
which in turn, as claimed, formally yields a commutative diagram of algebraic groups

 \begin{equation}\label{Diag:RepComp2}\xymatrix{1\ar[r]& G(\cP|_{\cX_{\bar s}},\omega_{s'})\ar[r] \ar@/_6pc/@{.>}[dddd]_{csp_{\bar \eta,\bar s} }&G(\cP|_{\cX_s},\omega_{s'})\ar[r]\ar@{_{(}->}[dd]&G( \langle\cP|_{\cX_s}\rangle_0,\omega_{s'})\ar[r]\ar[dd]\ar@/^7pc/@{.>}[dddd]^{(csp_{\bar \eta,\bar s})_0}&1\\
 & G(\cP|_{\cX'_{ s'}},\omega_{\bar s}) \ar[d]\ar[u]_\simeq\ar@/_3pc/[dd]_{R\psi_{f'}}& & & \\
 &G(\cP|_{\cX'},\omega_{s'})  \ar@{^{(}->}[r]&  G(\cP,\omega_{s'})\ar[r]&G( \langle\cP \rangle_0,\omega_{s'})\ar[r]&1\\
  & G(\cP|_{\cX'_{ \eta'}},\omega_{s'}) \ar[d]^\simeq \ar[u]_\simeq & & & \\
1\ar[r]& G(\cP|_{\cX_{\bar \eta}},\omega_{s'})\ar[r]&G(\cP|_{\cX_\eta}, \omega_{s'})\ar[r]\ar[uu]_\simeq&G( \langle\cP|_{\cX_\eta}\rangle_0,\omega_{s'})\ar[r]\ar[uu]_\simeq&1\\}\end{equation}
 For simplicity, we now omit fiber functors from the notation.

 \section{Proofs}\label{Sec:Proofs} \textit{}\\
 \noindent  Unless otherwise stated, in this Section   $k$ denotes  a field of characteristic $0$,  $S $    a  smooth  geometrically connected variety over $k$  with generic point $\eta$ and  $f:\cX\rightarrow S$   a morphism, separated and of finite type.

 \subsection{Proof of Theorem \ref{MT} and Corollary \ref{Cor:MT}} 

 \subsubsection{Recollection on artinian and noetherian abelian categories}\textit{}\\

\paragraph{}Let   $\cA$ be a an  artinian and noetherian abelian category.  Then,

  \begin{enumerate}[leftmargin=*, parsep=0cm,itemsep=0.2cm,topsep=0.2cm]
 \item   For every $A\in \cA$ and $\phi\in \SEnd_{\cA}(A)$, one has a $\phi$-stable  direct sum decomposition (Fitting lemma): $A\simeq A_{\phi,0} \oplus  A_{\phi,\infty}$, with the property that the induced morphism $\phi:A_{\phi,0}\rightarrow A_{\phi,0}$ is nilpotent and   $\phi:A_{\phi,\infty}\rightarrow A_{\phi,\infty}$ is an automorphism; explicitly  $A_{\phi,0}=\ker(\phi^n)$, for $n\gg 0$,   $A_{\phi,\infty}=\Sim(\phi^n)$, for $n\gg 0$. In particular, for every $A\in \cA$, $A$ is indecomposable in $\cA$ if and only if $\SEnd_{\cA}(A)$ is a  local ring and every $A\in \cA$ admits  a Krull-Schmidt decomposition: $A$ decomposes into a direct sum $A=\oplus_{1\leq i\leq r} A_i$ with $A_1,\dots, A_r\in \cA$ indecomposable and the  indecomposable  objects $A_1,\dots, A_r$ (counted with multiplicity) are unique up to isomorphism and called the Krull-Schmidt or indecomposable factors of $A$.  In particular, $A$ is semisimple if and only if its indecomposable factors are  simple. 

  \item   Every   $A\in \cA$   admits a composition series    that is  a filtration   $$0=A_{r+1}\subsetneq A_r\subsetneq \cdots \subsetneq A_2\subsetneq A_1:=A$$ in $\cA$ with $S_i:=A_i/A_{i+1}$ a simple object in $\cA$, $i=1,\dots, r$. Furthermore,  the  simple objects $S_1,\dots, S_r$ (counted with multiplicity) are unique up to isomorphism and called the Jordan-H\"{o}lder or simple factors of $A$. In particular the length $\hbox{\rm length}_{\cA}(A):=r$  of $A$ is a well-defined integer.   
  \item Let $\cN\subset \cA$ be a Serre subcategory and let $p:\cA\rightarrow\overline{\cA}:=\cA/\cN$ denote the resulting quotient functor, which is   exact and essentially surjective. Then,  for every simple object $S$ in $ \cA$ not lying in $\cN$, $p(S)$ is again a simple object in $\overline{\cA}$. This follows from the definition of morphisms in $\overline{\cA}$. Indeed, consider a  diagram
  $$ X\stackrel{f}{\rightarrow} Y \stackrel{s}{\hookleftarrow} S$$
  in $\cA$ with $N:=\hbox{\rm coker}(s)\in \cN$ such that the resulting morphism 
  $$p(X)\stackrel{p(f)}{\rightarrow} p(Y) \stackrel{p(s)^{-1}}{\stackrel{\simeq}{\rightarrow}} p(S)$$
  is injective. In particular, the morphism $p(S) \stackrel{p(s)}{\stackrel{\simeq}{\rightarrow}} p(Y)\rightarrow p(Y/X)$ is surjective. So either $p(Y/X)=0$ and $p(X)\stackrel{p(f)}{\rightarrow} p(Y) $ is an isomorphism in $\overline{\cA}$ or the morphism $S\stackrel{s}{\rightarrow} Y\rightarrow Y/X$ is non-zero hence injective. But then  the morphism $p(S) \stackrel{p(s)}{\stackrel{\simeq}{\rightarrow}} p(Y)\rightarrow p(Y/X)$ is an isomorphism in $\overline{\cA}$ hence so is $p(Y)\rightarrow p(Y/X)$, which imposes $p(X)=0$. In particular, for every $A\in \cA$, if one defines $\hbox{\rm length}_{\cA,\cN}(A)\leq \hbox{\rm length}_{\cA}(A)$ to be the number of  Jordan-H\"{o}lder factors of $A$ which lies in $\cN$, one has 
  $$ \hbox{\rm length}_{\overline{\cA}}(p(A))+ \hbox{\rm length}_{\cA,\cN}(A)= \hbox{\rm length}_{\cA}(A).$$
 \end{enumerate} 

 \paragraph{}\label{Par:PrelFunctor}Let now $\cA_1$, $\cA_2$ be artinian and noetherian abelian categories and let $F:\cA_1\rightarrow \cA_2$ be an additive functor.

  \begin{enumerate}[leftmargin=*, parsep=0cm,itemsep=0.2cm,topsep=0.2cm]
 \item Assume $F:\cA_1\rightarrow \cA_2$  is fully faithful.  Then for every $A_1\in \cA_1$, $A_1$ is indecomposable in $\cA_1$ if and only if $F(A_1)$ is indecomposable in $\cA_2$.
 \item Consider the following conditions  
  $$\begin{tabular}[t]{llll}
 (S$_{F,[??]}$)  & For every $A_1\in \cA_1$, $A_1$ is simple in $\cA_1$  &$[??]$ &$F(A_1)$ is simple  in $\cA_2$;\\
  (SS$_{F,[??]}$)  & For every $A_1\in \cA_1$, $A_1$ is semisimple in $\cA_1$  &$[??]$ &$F(A_1)$ is semisimple  in $\cA_2$. 
    \end{tabular}$$
    with $[??]$ one of $\Leftarrow$, $\Rightarrow$, $\Leftrightarrow$. 
    \begin{enumerate}[leftmargin=*, parsep=0cm,itemsep=0.2cm,topsep=0.2cm]
     \item Then  (S$_{F,\Rightarrow}$) always implies (SS$_{F,\Rightarrow}$) and, if  furthermore $F:\cA_1\rightarrow \cA_2$  is  exact, then for every $A_1\in \cA_1$, 
$$\hbox{\rm length}_{\cA_1}(A_1)=\hbox{\rm length}_{\cA_2}(F(A_1)).$$
 \item If   $F:\cA_1\rightarrow \cA_2$  is fully faithful then (S$_{F,\Leftarrow}$)  implies (SS$_{F,\Leftarrow}$). Indeed, let $A_1\in \cA_1$ and 
  consider the Krull-Schmidt decomposition $A_1=\oplus_{1\leq i\leq r} A_{1,i}$ of $A_1$ in $\cA_1$. Then $F(A_1)=\oplus_{1\leq i\leq r} F(A_{1,i})$ is the Krull-Schmidt decomposition of  $F(A_1)$ in $\cA_2$. So, one has 
 $$\begin{tabular}[t]{ll}
  $F(A_1)$ is semisimple in  $\cA_2$  &$\Longleftrightarrow$ $F(A_{1,i})$ is simple in $\cA_2$, $i=1,\dots r$\\
  &$\stackrel{\hbox{\rm \tiny (S$_{F,\Leftarrow}$)}}{\Longrightarrow}$ $A_{1,i}$ is simple in $\cA_1$, $i=1,\dots r$\\
    &$\Longleftrightarrow$ $A_1$ is semisimple in $\cA_1$.
    \end{tabular}$$
    \end{enumerate}
    \item   Assume $F:\cA_1\rightarrow \cA_2$  is exact. Let $\cN_2\subset \cA_2$ be a Serre subcategory and let $p_2: \cA_2\rightarrow \overline{\cA}_2:=\cA_2/\cN_2$ denote the resulting quotient functor. The full subcategory 
$$\cN_1:=\ker(\cA_1\stackrel{F}{\rightarrow}\cA_2\stackrel{p_2}{\rightarrow}\overline{\cA}_2)\subset \cA_1$$
is also a Serre subcategory and the   resulting quotient functor
 $p_1: \cA_1\rightarrow \overline{\cA}_1:=\cA_1/\cN_1 $ fits into a canonical commutative diagram of exact functors
 $$\xymatrix{ \cA_1\ar[r]^{p_1}\ar[d]_F& \overline{\cA}_1\ar[d]^{\overline{F}}\\
  \cA_2\ar[r]_{p_2}& \overline{\cA}_2.}$$
  As  $\overline{F}:\overline{\cA}_1\rightarrow \overline{\cA}_2$ is faithful exact,     (S$_{\overline{F},\Leftarrow}$) always  holds. 
    \begin{enumerate}[leftmargin=*, parsep=0cm,itemsep=0.2cm,topsep=0.2cm]
     \item Also,  (S$_{F, \Rightarrow}$) always implies  (S$_{\overline{F}, \Rightarrow}$). Indeed, for every $A_1\in \cA_1$, consider a  Jordan-H\"{o}lder filtration   $$\underline{A}_1: 0=A_{1,r+1}\subsetneq A_{1,r}\subsetneq \cdots \subsetneq A_{1,2}\subsetneq A_{1,1}:=A_1$$ of $A_1$ in $\cA_1$.
Assume   (S$_{F, \Rightarrow}$) holds.
 Then 
     $$F(\underline{A}_1): 0=F(A_{1,r+1})\subsetneq F(A_{1,r})\subsetneq \cdots \subsetneq F(A_{1,2})\subsetneq F(A_{1,1}):=F(A_1)$$ is again a Jordan-H\"{o}lder filtration of $F(A_1)$ in $\cA_2$. If $p_1(A_1)$ is simple in $\overline{\cA}_1$, then  all but one of the   $ A_{1,i}/  A_{1,i+1}$     lie in $\cN_1$ which, again by definition of $\cN_1$, ensures that all but one of the  $ F(A_{1,i})/  F(A_{1,i+1})$ lie in $\cN_2$, hence that    $\overline{F}(p_1(A_1))$ is simple in $\overline{\cA}_2$.
     
     \item  The argument in (3) (a) shows more precisely that for every  $A_1\in \cA_1$  with a  Jordan-H\"{o}lder filtration   $$\underline{A}_1: 0=A_{1,r+1}\subsetneq A_{1,r}\subsetneq \cdots \subsetneq A_{1,2}\subsetneq A_{1,1}:=A_1$$   in $\cA_1$, if   $ F(A_{1,i})/  F(A_{1,i+1})$  is a simple object in $\cA_2$, $i=1,\dots, r$,  then one has
$$\hbox{\rm length}_{\cA_1}(A_1)=\hbox{\rm length}_{\cA_2}(F(A_1)),\;\; \hbox{\rm length}_{\overline{\cA}_1}(p_1(A_1))=\hbox{\rm length}_{\overline{\cA}_2}(\overline{F}(p_1(A_1)))$$ 
and   \begin{itemize}[leftmargin=1cm, parsep=0cm,itemsep=0.2cm,topsep=0cm, label=]
 \item $A_1$ is semisimple in $\cA_1$ $\Rightarrow$ $F(A_1)$ is semisimple in $\cA_2$
 \item  $p_1(A_1)$ is semisimple in $\overline{\cA}_1$ $\Rightarrow$  $\overline{F}(p_1(A_1))$ is semisimple in $\overline{\cA}_2$.\\
 \end{itemize}  
      \end{enumerate}  
           \end{enumerate} 
\subsubsection{Preliminary reductions} \textit{}\\
\paragraph{\textit{Independence of the geometric point}}\label{Par:IndepGeoPt} The following observation  will enable us to choose geometric points freely. 
 Let $f:\cX\rightarrow S$ be a morphism, separated and of finite type. Let  $\cP \in   \Perv^{\ULA}(\cX/S)$. Let $t\in S$ and let $\bar t_1,\bar t_2$ be two geometric points over $t$. Then

  \begin{enumerate}[leftmargin=*, parsep=0cm,itemsep=0.2cm,topsep=0.2cm]
 \item  $\cP|_{\cX_{\bar t_1}}$ is simple (resp. semisimple) in $\Perv(\cX_{\bar t_1})$ if and only if $\cP|_{\cX_{\bar t_2}}$ is simple (resp. semisimple) in $\Perv(\cX_{\bar t_2})$ and one has
$$\hbox{\rm length}_{\Perv(\cX_{\bar t_1})}(\cP|_{\cX_{\bar t_1}})=\hbox{\rm length}_{\Perv(\cX_{\bar t_2})}(\cP|_{\cX_{\bar t_2}}).$$
  \item Assume furthermore that $f:\cX\rightarrow S$ is an abelian scheme. Then  $\cP|_{\cX_{\bar t_1}}$ is simple (resp. semisimple) in $P(\cX_{\bar t_1})$ if and only if $\cP|_{\cX_{\bar t_2}}$ is simple (resp. semisimple) in $P(\cX_{\bar t_2})$ and  one has
$$\hbox{\rm length}_{P(\cX_{\bar t_1})}(\cP|_{\cX_{\bar t_1}})=\hbox{\rm length}_{P(\cX_{\bar t_2})}(\cP|_{\cX_{\bar t_2}}).$$
 \end{enumerate}  
\noindent Indeed,  by considering a geometric point $\bar t$ over both $\bar t_1$ and $\bar t_2$ one immediately reduces to the case where, say, $\bar t_2$ is over $\bar t_1$. As the restrictions functors 
  $$-|_{\cX_{\bar t_2}}:\Perv(\cX_{\bar t_1}) \rightarrow \Perv(\cX_{\bar t_2}),\;\; -|_{\cX_{\bar t_2}}:P(\cX_{\bar t_1}) \rightarrow P(\cX_{\bar t_2})$$
are exact, fully faithful (\textit{e.g.} \cite[Lem. A.1]{JKLM}), the observations in Paragraph \ref{Par:PrelFunctor} (2) reduce the proof  of (1) and (2) to showing respectively that for every $\cP_1\in \Perv(\cX_{\bar t_1})$,
\begin{enumerate}[leftmargin=*, parsep=0cm,itemsep=0.2cm,topsep=0.2cm, label=(\arabic*)']
 \item   $\cP_1$ is simple in $\Perv(\cX_{\bar t_1})$ if and only if $\cP_1|_{\cX_{\bar t_2}}$ is simple in $\Perv(\cX_{\bar t_2})$;
 \item (if $f:\cX\rightarrow S$ is an abelian scheme) $\cP_1$ is simple in $P(\cX_{\bar t_1})$ if and only if $\cP_1|_{\cX_{\bar t_2}}$ is simple in $P(\cX_{\bar t_2})$
 \end{enumerate}  while the observation of Paragraph \ref{Par:PrelFunctor} (3) (a) shows that  Assertion  (2)'  follows from Assertion  (1)'. Let us prove Assertion  (1)'.  The if part of  follows from the fact that  $-|_{\cX_{\bar t_2}}:\Perv(\cX_{\bar t_1}) \rightarrow \Perv(\cX_{\bar t_2})$ is exact and fully faithful. For the only if part, from
  \cite[Thm. 4.3.1 (ii)]{BBD} every simple object $\cS$  in $\Perv(\cX_{\bar t_1})$ is of the form $\cS=\iota_{1*}j_{1!*}\cF_1[d]$ for some 
 irreducible closed subscheme $\iota_1: Z _1\hookrightarrow \cX_{\bar t_1}$,  non-empty open subscheme $j_1:U_1 \hookrightarrow Z_1$,   smooth over $k(\bar t_1)$ and pure of dimension $d$, and  simple $\overline{\Q}_\ell$-local system $\cF_1$ on $U_1$. Consider the base-change diagram
 $$\xymatrix{ U _2\ar@{^{(}->}[r]^{j_2}\ar[d]\ar@{}[dr]|\square\ar[d]&Z _2\ar@{^{(}->}[r]^{\iota_2}\ar@{}[dr]|\square\ar[d]& \cX_{\bar t_2}\ar[r]\ar@{}[dr]|\square\ar[d]&\spec(k(\bar t_2))\ar[d]\\
 U _1\ar@{^{(}->}[r]^{j_1}&Z _1\ar@{^{(}->}[r]^{\iota_1}& \cX_{\bar t_1}\ar[r]&\spec(k(\bar t_1))}$$
 and set $\cF_2:=\cF_1|_{U_2}$. Then $$\cS|_{\cX_{\bar t_2}}= (\iota_{1*}j_{1!*}\cF_1[d])|_{\cX_{\bar t_2}}\simeq \iota_{2*}j_{2!*}\cF_2[d]$$ and the assertion follows from the fact that  the restriction functor 
  $$-|_{U_2}:\Loc(U_1) \rightarrow \Loc(U_2)$$
maps simple objects to simple objects since 
  the canonical morphism $\pi_1(U_2)\rightarrow \pi_1(U_1)$ is surjective (\textit{e.g.}  \cite[\href{https://stacks.math.columbia.edu/tag/0387}{Tag 0387}]{stacks-project}).\\

\paragraph{}  Let $f:\cX\rightarrow S$ be a morphism, separated and of finite type. Every witness  $(S',\eta',s')\rightarrow (S,\eta,s)$ induces a canonical exact functor 
$$R\psi_{f'}: \Perv(\cX_{\eta'}) \stackrel{\simeq}{\rightarrow} \Perv^{\ULA}(\cX'/S')\rightarrow \Perv(\cX_{s'}),$$
 where the notation are as follows
$$\xymatrix{\cX'\ar[r]\ar[d]_{f'}\ar@{}[dr]|\square&\cX\ar[d]^f\\
S'\ar[r]& S.}$$ 

\noindent Assume furthermore $f:\cX\rightarrow S$ is an abelian scheme. Then  
 $$N(\cX_{\eta'})\cap \Perv(\cX_{\eta'})=\ker(\Perv(\cX_{\eta'})\stackrel{R\psi_{f'}}{\rightarrow}\Perv(\cX_{s'})\rightarrow P(\cX_{s'})).$$
\noindent So,  Paragraph \ref{Par:IndepGeoPt}  and the  observations in Paragraph \ref{Par:PrelFunctor} (2) (b) applied to 
 $$\xymatrix{ \cA_1\ar[r]^p\ar[d]_F& \overline{\cA}_1\ar[d]^{\overline{F}}&=&\Perv(\cX_{\eta'})\ar[r]\ar[d]_{R\psi_{f'}}&P(\cX_{\eta'})\ar[d] \\
  \cA_2\ar[r]_{p'}& \overline{\cA}_2&&  \Perv(\cX_{s'})\ar[r] & P(\cX_{s'}).}$$
reduce the proof of Theorem \ref{MT} and Corollary \ref{Cor:MT} to the following statement.

\begin{theoreme}\label{MTBis} Let $f:\cX\rightarrow S$ be a morphism, separated and of finite type.  Let  $\cP_i \in   \Perv(\cX_{\bar \eta})$, $i=1,\dots, r$  be finitely many simple objects in $ \Perv(\cX_{\bar \eta})$.  After possibly replacing $S$ by a  non-empty open subscheme $  S$ the following holds. For every $s\in S$, there exists a witness $(S',\eta',s')\rightarrow (S,\overline{\eta},s)$ such that 
$R\psi_{f'}(\cP_i|_{\cX_{\eta'}})$ is simple in $\Perv(\cX_{s'})$, $i=1,\dots, r$. \end{theoreme}

\subsubsection{Proof of Theorem \ref{MTBis}} \textit{}\\

\paragraph{\textit{Intermediate extensions and the ULA property}}\label{!* is ULA} Let $\ell$ be a prime. Let $\cZ\to S$ be a separated morphism of finite presentation of quasi-compact quasi-separated schemes over $\Z[1/\ell]$. Assume that  $S$ has only finitely many irreducible components, so that by \cite[Theorem 6.7]{HansenScholze} the relative perverse t-structure exists on $D^{\ULA}(\cZ/S)$. Let $j:\cU\hookrightarrow \cZ$ be an open immersion of finite presentation.
 	Given $\cK\in D^{\ULA}(\cU/S)$ with $j_!\cK,\, j_!D_{\cU/S}\cK\in  D^{\ULA}(\cZ/S)$, write $j_{!*/S}\cK$ for the image of the natural morphism \[{}^{p/S}H^0(j_!\cK)\to {}^{p/S}H^0(D_{\cZ/S}j_!D_{\cU/S}\cK)\] in the abelian category $\Perv^{\ULA}(\cZ/S)$.  Observe the following
 \begin{enumerate}[leftmargin=*, parsep=0cm,itemsep=0.2cm,topsep=0.2cm]
 \item When $S$ is the spectrum of a field and $\cK\in \Perv(\cU)$,   $j_{!*/S}\cK$ is  the usual middle extension of $\cK$ to $\cZ$.
\item   As, for ULA objects, both the formation of relative Verdier duality and $j_!$ commute with base change $S'\to S$ (\cite[Proposition 3.4 (ii)]{HansenScholze}), the formation of $j_{!*/S}\cK$ also commutes with base-changes $S'\to S$. In particular, 
for every geometric point $\bar{s}$ on $S$, if $j_{\bar s}:\cU_{\bar s}\hookrightarrow \cZ_{\bar s}$ denotes the base change of $j:\cU\hookrightarrow \cZ$ along $\bar s\rightarrow S$, one has $(j_{!*/S}\cK)|_{\cZ_{\bar s}}=j_{\bar{s},!*}(\cK|_{\cU_{\bar s}})$. 
	\end{enumerate}
	
	 \begin{lemme}\label{j!ULA}
 	Consider a  diagram \begin{equation}\label{relative alteration}\xymatrix{\widetilde{\cU}\ar@{^{(}->}[r]^{\widetilde{j}}\ar@{}[dr]|\square\ar[d]_h& \widetilde{\cZ} \ar[d]^g& \\
 			\cU \ar@{^{(}->}[r]^{j}   & \cZ \ar[r]^f & S }\end{equation}
  of quasi-compact quasi-separated schemes. Assume that \begin{itemize}[leftmargin=*, parsep=0cm,itemsep=0.2cm,topsep=0.2cm]
 		\item $f:\cZ\to S$ is separated of finite presentation and $j:\cU\hookrightarrow \cZ$ is an open immersion with $\cU\to S$  smooth, ;
		\item 
$g: \widetilde{\cZ}\to \cZ$ is proper,  and $h: \widetilde{\cU}\to \cU$ is finite étale;
 		\item $ \widetilde{\cZ}\to S$ is smooth and  $ \widetilde{\cD}:=\widetilde{\cZ}\setminus \widetilde{\cU}$ is a divisor on  $ \widetilde{Z}$ with strict normal crossings relative to $ \widetilde{\cZ}\to S$.
 	\end{itemize} 
 	Let $\ell$ be a prime invertible on $S$. 
 	Let $\cF $ be a $\overline{\Q}_\ell$-local system on $\cU$. Assume that $h^*\cF$ is tamely ramified along $ \widetilde{\cD}$. 
 	Then $j_!\cF$, $j_!(\cF^\vee)$ are  constructible sheaves that are ULA relative to $f:\cZ\to S$.  
 	In particular, if $S$ has only finitely many irreducible components and $d $ denotes the relative dimension of $\cU\to S$, then  $j_{!*/S}\cF[d ]$ is a well defined object in $ \Perv^{\ULA}(\cZ/S)$, whose formation commutes with base-changes. 
 \end{lemme}
 
 \begin{proof} By \cite[\href{https://stacks.math.columbia.edu/tag/0818}{Tag 0818}]{stacks-project},  $j:\cU\hookrightarrow \cZ$ is of finite presentation. From \cite[Proposition 1.4.4]{laumon1981semi}, and as $\cU\to S$ is smooth,   $\cF\in D^{\ULA}(\cU/S)$. So, since
 	$D_{\cU/S}(\cF[d])=\cF^\vee[d]$, the second part of Lemma \ref{j!ULA}   follows from the first part. 
 
The sheaves $j_!\cF$ and $\widetilde{j}_!h^*\cF$ are constructible. 
 	 From \cite[Lemma 3.14]{saito2013wild}, 
	and as $h^*\cF$ is tamely ramified along $ \widetilde{\cD}$, $\widetilde{j}_!h^*\cF$ is ULA relative to $ \widetilde{\cZ}\to S$. From \cite[p.643]{HansenScholze}, and as $g: \widetilde{\cZ}\to \cZ$ is proper, $Rg_*\widetilde{j}_!h^*\cF\in D^{\ULA}(\cZ/S)$. Hence  $j_!h_*h^*\cF\simeq Rg_*\widetilde{j}_!h^*\cF\in D^{\ULA}(\cZ/S)$.
 	
 	 Since $h:\widetilde{\cU}\to \cU$ is finite étale, the natural morphism $\overline{\Q}_{\ell,\cU}\to h_*\overline{\Q}_{\ell,\widetilde{\cU}}$ of lisse sheaves is the inclusion of a direct summand. By the projection formula \cite[\href{https://stacks.math.columbia.edu/tag/0F0G}{Tag 0F0G}]{stacks-project}, as $h:\widetilde{\cU}\to \cU$ is proper, one has $h_*(h^*\cF)=\cF\otimes_{\overline{\Q}_\ell} h_*\overline{\Q}_{\ell,\widetilde{\cU}}$.  Thus, $\cF\hookrightarrow h_*h^*\cF$ and hence $j_!\cF\hookrightarrow j_!h_*h^*\cF$ are  direct summands. Since universal local acyclicity is preserved
 	under  passing to direct summands, one has $j_!\cF\in D^{\ULA}(\cZ/S)$. 
 	As $h^*\cF^\vee$ is also tamely ramified along $\widetilde{\cD}$, one has  $j_!\cF^\vee\in D^{\ULA}(\cZ/S)$. 
 	
 \end{proof}
\textit{}\\
\paragraph{Proof of Theorem \ref{MTBis}} 

From \cite[Thm. 4.3.1 (ii)]{BBD}, for every $i=1,\dots, r$,  there exists an integral closed subscheme $\iota_i: Z_i \hookrightarrow \cX_{\bar \eta}$ and a non-empty open subscheme $j_i:U_i \hookrightarrow Z_i$,   smooth over $k(\bar \eta)$ and pure of dimension $d_i$, and a simple object  $\cF_i$ in $\Loc(U_i)$ such that $\cP_i=\iota_{i,*}j_{i,!*}\cF_i[d_i]$. Fix also a $k(\bar \eta)$-point $u_i\in U_i$ and a smooth normal crossing compactification $U_i\hookrightarrow U^{\cpt}_i$. There exists a finite field extension $K_0$ of $k(\eta)$ such that, for every $i=1,\dots, r$,   $$\xymatrix{U^{\cpt}_i&U_i\ar@{_{(}->}[l] \ar@{^{(}->}[r]^{j_i}&Z_i \ar@{^{(}->}[r]^{\iota_i}& \cX_{\bar \eta}\ar[r]& \spec(k(\bar \eta))\ar@/_1.5pc/[lll]_{u_i}}$$ is defined over $K_0$ and spread out as 
  $$\xymatrix{U_i^{\cpt}\ar@{}[dr]|\square\ar[d]&U_i\ar@{^{(}->}[r]^{j_i}\ar[d]\ar@{}[dr]|\square\ar@{_{(}->}[l]&Z_i\ar@{^{(}->}[r]^{\iota_i}\ar[d]\ar@{}[dr]|\square&\cX_{\eta_0}\ar[d]\ar[r]&\spec(K_0)\ar[d]^{\eta_0}\ar@/_1.5pc/[lll]_{u_i}\\
\cU_i^{\cpt}& \cU_i\ar@{^{(}->}[r]^{j_i}\ar@{_{(}->}[l]&\cZ_i\ar@{^{(}->}[r]^{\iota_i}&\cX\times_SS_0\ar[r]&S_0\ar@/^1.5pc/[lll]^{u_i}}$$
with  
 \begin{itemize}[leftmargin=*, parsep=0cm,itemsep=0.2cm,topsep=0.2cm, label=-]
 \item $\iota_i: \cZ_i\hookrightarrow \cX\times_SS_0$  a closed immersion and  $\cZ_i\rightarrow S_0$  geometrically irreducible;
 \item $j_i:\cU_i\hookrightarrow \cZ_i$ an open immersion and $\cU_i\rightarrow S_0$ is smooth, pure of relative dimension $d_i$;
 \item  $\cU_i\hookrightarrow \cU_i^{\cpt}$ is a relative smooth normal crossing compactification over $S_0$,
 \end{itemize}  
where $S_0\subset \widetilde{S}_0$ is a non-empty open in the normalization $ \widetilde{S}_0\rightarrow S$ of $S$ in  $\spec(K_0)\rightarrow \spec(k(\eta))\stackrel{\eta}{\rightarrow}S$. 

By Lemma \ref{Hironaka}, as $\mathrm{char} (k)=0$, shrinking $S_0$ if necessary,  one may furthermore assume  that  there is a diagram \begin{center}
	$$\xymatrix{
		\widetilde{\cU}_i  \ar@{^{(}->}[r]      \ar[d] \ar@{}[dr]|\square& \widetilde{\cZ}_i \ar[d] \\
		\cU_i  \ar@{^{(}->}[r]^{j_i}                             & \cZ_i \ar[r] &  S_0}$$
	 
\end{center}satisfying the conditions of Lemma \ref{j!ULA}. Then the diagram  base changed along a witness $S'\to S_0$ also satisfies the conditions of Lemma \ref{j!ULA}. From $\mathrm{char} (k)=0$, the tame ramification condition holds.

As the image of every non-empty open subscheme of $S_0$ contains a non-empty open subscheme of $S$ and as, for every $s_0\in S_0$ with image $s\in S$ any witness $(S_0',\eta_0',s_0')\rightarrow (S_0,\eta_0,s_0)$ induces a witness $(S_0',\eta_0',s_0')\rightarrow (S_0,\eta_0,s_0)\rightarrow (S,\eta,s)$, one may freely replace  $S$ with $S_0$  so that we remove the subscripts $(-)_0$ from the notation. Let now $s\in S$  and fix a witness   $(S' ,\eta',s')\rightarrow (S,\overline{\eta},s)$. By invariance of \'etale fundamental group under extensions of algebraically closed fields in characteristic $0$, one has $\pi_1(\cU_{i,\eta'})\tilde{\rightarrow}\pi_1(U_i)$ hence $\cF_i|_{\cU_{i,\eta'}}$ is again irreducible.  As $S'$ is strictly henselian,  $\pi_1(S')=1$  and as $\cU_i':=\cU_i\times_SS'\rightarrow S'$ has a section, the canonical morphisms 
 $$ \pi_1(\cU'_{i,\eta'})\tilde{\rightarrow} \pi_1(\cU'_i) \tilde{\leftarrow} \pi_1(\cU'_{i,s'})$$
 are both isomorphisms  \cite[XIII, 4.3, 4.4]{SGA1}. In particular,   $\cF_i|_{\cU_{i,\eta'}}$  extends uniquely to an object   $\cF_i'$ in $\Loc(\cU_i') $, and  $\cF_i'|_{\cU'_{i,s'}}$ is simple in $\Loc(\cU'_{i,s'})$.  From \cite[Thm. 4.3.1 (ii)]{BBD}, it is thus enough to show that 
 \begin{equation}\label{Rpsi !*}R\psi_{f'}(\cP_i|_{\cX_{\eta'}})\simeq \iota_{i,s'*}j_{i,s'!*}(\cF'_i|_{\cU'_{i,s'}}[d_i]).\end{equation}
 
 Consider the commutative diagram \begin{center}
 $$	\xymatrix{
 		\cU'_i \ar@{^{(}->}[r]^{j_i'}       \ar[d] \ar@{}[dr]|\square & \cZ'_i \ar@{^{(}->}[r]^{\iota_i'}    \ar[d] \ar@{}[dr]|\square& \cX' \ar[r]^{f'}\ar[d] \ar@{}[dr]|\square& S' \ar[d] \\
 		\cU_i \ar@{^{(}->}[r]^{j_i}                                                 & \cZ_i\ar@{^{(}->}[r]^{\iota_i}                                                    & \cX \ar[r]^f                                       & S          }$$
 \end{center} of schemes with cartesian squares. By Lemma \ref{j!ULA}, $j'_{i,!*/S'}\cF'_i[d_i]\in \Perv^{\ULA}(\cZ'_i/S')$. Therefore, $\cK:=\iota'_{i*}j'_{i,!*/S'}\cF'_i[d_i]$ is in $\Perv^{\ULA}(\cX'/S')$. By Observation (2) in Paragraph \ref{!* is ULA} and the proper base change theorem, as $\iota'_i:\cZ'_i\hookrightarrow \cX'$ is proper, $\cK|_{\cX_{\eta'}}$ is the pullback of $\cP_i$ along $\cX_{\eta'}\to \cX_{\bar\eta}$, and  $\cK|_{\cX_{s'}}$ is $\iota_{i,s'*}j_{i,s'!*}(\cF'|_{\cU'_{i,s'}}[d_i])$, which proves \eqref{Rpsi !*}.

 \begin{lemme}\label{Hironaka}
 Let $S$ be an irreducible scheme with generic point $\eta$. Assume that $k(\eta)$ is of characteristic $0$. Let $f:\cZ\to S$ be a morphism separated of finite presentation, with $\cZ_{\eta}$ integral. Let $U \hookrightarrow \cZ_{\eta}$ be an  open subset smooth over $k(\eta)$. Then up to shrinking $S$ to an affine open subset, there is a diagram \eqref{relative alteration} satisfying the conditions of Lemma \ref{j!ULA},  such that $\cU_{\eta}=U$ and $h:\widetilde{\cU}\to \cU$ is an isomorphism. \end{lemme}\begin{proof}
By Hironaka's resolution of singularities (see, e.g., \cite[I, 3.1.5 b) a)]{SGA5}), as $k(\eta)$ is of characteristic $0$, $\cZ_{\eta}$ is strongly desingularizable. As $\cZ_{\eta}$ is integral, there is a proper morphism $ \widetilde{Z}\to \cZ_{\eta}$ with $\widetilde{Z}$ smooth over $k(\eta)$, such that the pullback  $\widetilde{U}:= U\times_{ \cZ_\eta}\widetilde{Z}\rightarrow U $ is an isomorphism, and that $\widetilde{Z}\setminus \widetilde{U}$ is a strict normal crossing divisor. The result follows by spreading out.
\end{proof}
 
  \begin{remarque}\label{semistable RPsi}\textnormal{The nearby cycles functor may not preserve simplicity of perverse sheaves nor commute with  middle extension.
  	Let $S$ be the spectrum of a strictly Henselian discrete valuation ring with generic point $\eta$ and closed point $s$.  Let $f:\cX\to S$ be a proper semi-stable morphism with geometrically integral fibers of dimension $d$. Assume that the special fiber $\cX_s\hookrightarrow \cX $ is a strict normal crossing divisor on $\cX$. Then there is an open subset $j:\cU\hookrightarrow \cX$ smooth over $S$, such that $\cU_{\eta}=\cX_{\eta}$ and that $\cU_s$ is Zariski-dense in $\cX_s$.   	Let $R\psi_f:D_c^b(\cX_{\bar \eta})\to D_c^b(\cX_s)$ be the nearby cycles functor.   	By \cite[Théorème 3.2 (c) (i)]{illusie1994autour}, 
 $H^0R\psi_f(\overline{\Q}_{\ell,\cX_{\bar \eta}})\simeq \overline{\Q}_{\ell,\cX_s}$. Let $j_s:\cU_s\hookrightarrow \cX_s$ be the pullback of $j:\cU\hookrightarrow \cX$ along $s\to S$. Let $\IC_{\cX_s}:=j_{s,!*}\overline{\Q}_{\ell,\cU_s}[d]$ be the intersection cohomology complex on $\cX_s$. In general, $H^{-d}\IC_{\cX_s}$ is not constant, in which case the perverse sheaf  $R\psi_f(\overline{\Q}_{\ell,\cX_{\bar \eta}}[d])$ is not isomorphic to $\IC_{\cX_s}$. Also, from \cite[XV, Thm 2.1]{SGA4-3}, one has \[\left(R\psi_f(\overline{\Q}_{\ell,\cX_{\bar \eta}})\right)|_{\cU_s}=R\psi_{f\circ j}(\overline{\Q}_{\ell,\cX_{\bar \eta}})=\overline{\Q}_{\ell,\cU_s}.\]
  	Then by \cite[Théorème 4.3.1 (ii)]{BBD},  $R\psi_f(\overline{\Q}_{\ell,\cX_{\bar \eta}}[d])$ is not simple in $\Perv(\cX_s)$ (otherwise, it would be isomorphic to the simple object $\IC_{\cX_s}$) while $\overline{\Q}_{\ell,\cX_{\bar\eta}}[d]$ is simple in $\Perv(\cX_{\bar\eta})$.}
  \end{remarque}

 \subsection{Lifting semisimplicity}\label{Sec:LiftingSS} 
 \subsubsection{}Let $\cA$ be an artinian and noetherian abelian category, let $\iota:\cN\hookrightarrow\cA$ be a Serre subcategory and let $p:\cA\rightarrow \overline{\cA}:=\cA/\cN$ denote the resulting quotient functor. The inclusion functor  $\iota:\cN\hookrightarrow\cA$   admits both a right adjoint $(-)_{\neg}:\cN\rightarrow \cA$ ("maximal  subobject in $\cN$") and a left adjoint $(-)^{\neg}:\cN\rightarrow \cA$  ("maximal   quotient object in $\cN$"). Explicitly, for $A\in \cA$, $A_{\neg}=\sum_{N\in S_{\neg}(A)} N\hookrightarrow  A$, where $S_{\neg}(A)$ denotes the subset of all   subobjects of $A$ in $\cN$ and $A\twoheadrightarrow A^{\neg}=A/\cap_{S\in S^{\neg}(A)} S$, where $S^{\neg}(A)$ denotes the     subobjects $S$ of  $A$ such that $A/S\in\cN$. Then for every $A\in \cA$,
 $$A^*:=\ker(A/A_{\neg}\rightarrow (A/A_{\neg})^{\neg})$$
 is a subquotient of $A$ in $\cA$ satisfying  $(A^*)^{\neg}=(A^*)_{\neg}=0$ and $p(A^*)\simeq p(A)$ in $\overline{\cA}$. Observing that for every $A_1,A_2\in \cA$ with $A_1^{\neg}=0$ and $A_{2,\neg}=0$ the canonical morphism
 $$\SHom_{\cA}(A_1,A_2)\rightarrow \SHom_{\overline{\cA}}(p(A_1),p(A_2))$$
 is an isomorphism, one gets that for every $A_1,A_2\in \cA$,  
 $$\begin{tabular}[t]{l} $p(A_1)\simeq p(A_2)$ in $\overline{\cA}$ if and only if $A_1^*\simeq A_2^*$ in $\cA$.\end{tabular}$$
  \begin{lemme}\label{Lem:FormalSS}  Let $A\in \cA$ such that $p(A)$ is semisimple in $\overline{\cA}$. Then $A^*$ is semisimple in $\cA$ and 
  $$\hbox{\rm length}_{\cA}(A^*)=\hbox{\rm length}_{\overline{\cA}}(p(A)).$$
 \end{lemme}
 \begin{proof} Assume first $p(A)$ is  simple in $\overline{\cA}$. As $p(A^*)\simeq p(A)$ in $\overline{A}$, $A^*$  has a single   Jordan-H\"{o}lder factor in $\cA$ which is not in $\cN$. But as $(A^*)_{\neg}=(A^*)^{\neg}=0$, this forces $A^*$ to be simple in $\cA$. In general,  let $\overline{S}_1,\dots, \overline{S}_r$ denote the simple factors (counted with multiplicities) of $p(A)$ in $\overline{\cA}$ and let $S_1,\dots, S_r\in \cA$ with $p(S_i)\simeq \overline{S}_i$ in $\overline{\cA}$, $i=1,\dots, r$. Set $$S:=S_1^*\oplus\cdots \oplus S_r^*.$$
 Then, $S$ is  semisimple   in $\cA$,   $S=S^*$, and   $p(S)\simeq p(A)$ in $\overline{\cA}$. This shows $A^*\simeq (S^*\simeq) S$ is semisimple in $\cA$ and 
 $$\hbox{\rm length}_{\cA}(A^*)=\hbox{\rm length}_{\cA}(S)=r= \hbox{\rm length}_{\overline{\cA}}(p(A)).$$
 \end{proof}

 \subsubsection{}If $X$ is an abelian variety over a field $K$ of characteristic $0$,  write $$(-)_{\neg}:\Perv(X)\rightarrow N(X)\cap \Perv(X),\;\;   (-)^{\neg}:\Perv(X)\rightarrow N(X)\cap \Perv(X)$$ for the  right adjoint ("maximal negligible subobject") and left adjoint ("maximal negligible quotient object")  of  the inclusion functor 
 $$ N(X)\cap \Perv(X)\hookrightarrow  \Perv(X)$$ respectively. By Galois descent \cite[Lem. A.6]{Timo}, for every $\cP\in \Perv(X)$, $(\cP_{\neg})|_{X_{\bar K}}= (\cP|_{X_{\bar K}})_{\neg}\hookrightarrow \cP|_{X_{\bar K}}$ and $\cP|_{X_{\bar K}}\twoheadrightarrow (\cP^{\neg})|_{X_{\bar K}}= (\cP|_{X_{\bar K}})^{\neg}$; in particular
 
 \begin{equation}\label{Diag:Commute1}(-)^*\circ -|_{X_{ \bar K}}\simeq  -|_{X_{\bar K}}\circ (-)^*:  \Perv (X)\rightarrow  \Perv(X_{\bar K}).\end{equation}

 \noindent If $f:\cX\rightarrow S$ is an abelian scheme, write again \begin{gather*}
 	(-)_{\neg}:\Perv^{\ULA}(\cX/S)\rightarrow N^{\ULA}(\cX/S)\cap \Perv^{\ULA}(\cX/S),\\
 	 (-)^{\neg}:\Perv^{\ULA}(\cX/S)\rightarrow N^{\ULA}(\cX/S)\cap \Perv^{\ULA}(\cX/S) \end{gather*}
 for the right adjoint and left adjoint of the inclusion functor
 $$N^{\ULA}(\cX/S)\cap \Perv^{\ULA}(\cX/S) \hookrightarrow   \Perv^{\ULA}(\cX/S) $$
 respectively.   Furthermore, as for every  $s\in S$, 
 $$ N^{\ULA}(\cX/S)\cap \Perv^{\ULA}(\cX/S)=\ker( \Perv^{\ULA}(\cX/S)\rightarrow P(\cX_{\bar s})), $$ 
and, for $s=\eta$, $-|_{\cX_\eta}: \Perv^{\ULA}(\cX/S)\rightarrow \Perv(\cX_\eta)$ is fully faithful with essential image stable under subquotients, one has

 \begin{equation}\label{Diag:Commute2}(-)^*\circ -|_{\cX_{ \eta}}\simeq  -|_{\cX_{  \eta}}\circ (-)^*:  \Perv^{\ULA}(\cX/S)\rightarrow  \Perv(\cX_{ \eta}).\end{equation}
 
\noindent Combining (\ref{Diag:Commute1}),  (\ref{Diag:Commute2}) one gets 

$$(-)^*\circ -|_{\cX_{ \bar \eta}}\simeq  -|_{\cX_{ \bar  \eta}}\circ (-)^*:  \Perv^{\ULA}(\cX/S)\rightarrow  \Perv(\cX_{\bar  \eta}).$$
 
\noindent This observation together  with Lemma \ref{Lem:FormalSS} yields the following result.
 \begin{corollaire}\label{Cor:LiftingSS}  Let   $f:\cX\rightarrow S$ an abelian scheme.  Let $\cP\in \Perv^{\ULA}(\cX/S)$. 
 Assume that $\cP|_{\cX_{\bar \eta}}$ is semisimple in $P(\cX_{\bar \eta})$.    Then   $\cP^*|_{\cX_{\bar \eta}}$ is  semisimple in $\Perv(\cX_{\bar \eta})$ with
  $$\hbox{\rm length}_{\hbox{\rm \small  Perv}(\cX_{\bar \eta})}(\cP^*|_{\cX_{\bar \eta}})=\hbox{\rm length}_{P(\cX_{\bar \eta})}(\cP|_{\cX_{\bar \eta}}).$$
   and, for every $s\in S$, $\cP^*|_{\cX_{\bar s}}\simeq \cP|_{\cX_{\bar s}}$ in  $P(\cX_{\bar s})$.
 \end{corollaire}

 \subsection{Proof of  Corollary \ref{MC2}}\label{Sec:ProofMC2}
Let   $f:\cX\rightarrow S$ an abelian scheme.  Let $\cP\in \Perv^{\ULA}(\cX/S)$ such that  $\cP|_{\cX_{\bar \eta}}$ is semisimple in $P(\cX_{\bar \eta})$.    From Corollary \ref{Cor:LiftingSS}, up to replacing $\cP$ with $\cP^*$, one may assume $\cP|_{\cX_{\bar \eta}}$ is semisimple in $\Perv(\cX_{\bar \eta})$ and, 
from   Theorem \ref{MT}, up to replacing $S$ by a non-empty open subscheme,  one may assume    that  for every $s\in S$,  $\cP|_{\cX_{\bar s}}$ is semisimple     in $\Perv(\cX_{\bar s})$. This reduces the proof of  Corollary  \ref{MC2} (2) to the one of  Corollary  \ref{MC2} (1). Under the assumptions of   Corollary  \ref{MC2} (1), $G (\cP|_{\cX_{\bar \eta}})$ is a reductive group and, for every $s\in S$,  $G (\cP|_{\cX_{\bar s}})\subset G (\cP|_{\cX_{\bar \eta}})$ is a closed reductive subgroup.   Recall that \cite[Prop. 3.1 (c)]{DeligneHC}  for a reductive group $G$ over a field $Q$ of characteristic $0$, a finite-dimensional $Q$-rational faithful representation $V$ of $G$ and a   closed reductive subgroup $H\subset G$,   one has 
$$H=\hbox{\rm Fix}_G(u_H)\subset G,$$ for some  integers $m,n\geq 0$ and $0\not= u_H\in T^{m,n}(V)$. In particular,   $H\subsetneq G$ if and only if 
 
$$\Sdim_Q(I^{m,n}(V))<\Sdim_Q(I^{m,n}(V|_H)),$$ for some  integers $m,n\geq 0$.\\

 \noindent This reduces the proof of Corollary \ref{MC2} (1) to the following.

\begin{corollaire}\label{Cor:MC2Bis} Let   $f:\cX\rightarrow S$ an abelian scheme.  Let $\cP\in \Perv^{\ULA}(\cX/S)$ such that  $\cP|_{\cX_{\bar s}}$ is semisimple in $\Perv(\cX_{\bar s})$ for  every $s\in S$. Then   for every integers $m,n\geq 0$, there exists  a strict  closed subscheme $S_{m,n}\hookrightarrow S$ such that for every   $s\in S$,  $$\Sdim_{\overline{\Q}_\ell}( I^{m,n}( \cP|_{\cX_{\bar \eta}}))<\Sdim_{\overline{\Q}_\ell}( I^{m,n}( \cP|_{\cX_{\bar s}}))$$   if and only if $ s\in S_{m,n}$.
\end{corollaire}

 \begin{proof}  Write $$\cP_{m,n}:={}^{p/S}\SH^0(T^{m,n}(\cP)),$$
  which is again in $\Perv^{\ULA}(\cX/S)$ with the properties that,
for  every $s\in S$,  $$ {}^p\SH^0(T^{m,n}(\cP|_{\cX_{\bar s}}))\simeq \cP_{m,n}|_{\cX_{\bar s}}$$ and, by Lemma \ref{Lem:TensorSS}, 
 $ \cP_{m,n}|_{\cX_{\bar s}}$  is semisimple in $\Perv(\cX_{\bar s})$.  For $s\in S$, decompose $ \cP_{m,n}|_{\cX_{\bar s}}$ as
 $$ \cP_{m,n}|_{\cX_{\bar s}}\simeq  (\cP_{m,n}|_{\cX_{\bar s}})_{\neg}\oplus \cS,$$
 where  $\cS$  is the sum of all  simple non-negligible subobjects of $\cP_{m,n}|_{\cX_{\bar s}}$ in $\Perv(\cX_{\bar s})$. 
    As for every $ \cN \in N(\cX_{\bar s})\cap \Perv(\cX_{\bar s})$ one has 
    $$\SHom_{\hbox{\rm \tiny Perv}(\cX_{\bar s})}( \cN ,\delta_0)=0,$$
  the canonical morphism
 $$\SHom_{\hbox{\rm \tiny Perv}(\cX_{\bar s})}(\cS,\delta_0)\rightarrow \SHom_{\hbox{\rm \tiny Perv}(\cX_{\bar s})}(\cP_{m,n}|_{\cX_{\bar s}},\delta_0)$$
is an isomorphism. On the other hand, as for every non-negligible simple objects $\cS_1,\cS_2$ in $\Perv(\cX_{\bar s})$ the canonical morphism   
$$\SHom_{\hbox{\rm \tiny Perv}(\cX_{\bar s})}(\cS_1,\cS_2)\rightarrow \SHom_{P(\cX_{\bar s})}(\cS_1,\cS_2)$$
is   an isomorphism,
 the canonical morphism
$$\SHom_{\hbox{\rm \tiny Perv}(\cX_{\bar s})}(\cS,\delta_0)\rightarrow \SHom_{P(\cX_{\bar s})}(\cS,\delta_0)$$
is also an isomorphism. This proves that 
$$\Sdim_{\overline{\Q}_\ell}( I^{m,n}( \cP|_{\cX_{\bar s}}))=\Sdim_{\overline{\Q}_\ell}( I^{1,0}( \cP_{m,n}|_{\cX_{\bar s}}))=\Sdim_{\overline{\Q}_\ell}(\SHom_{\hbox{\rm \tiny Perv}(\cX_{\bar s})}(\cP_{m,n}|_{\cX_{\bar s}},\delta_0)).$$
By Lemma \ref{MaximalSupport}, there  is a  quotient $\cP_{m,n}\twoheadrightarrow\cP_{m,n,\lbrace 0\rbrace} $
 in $\Perv(\cX/S)$ such that   for every geometric point $\bar s$ on $ S$, $ \cP_{m,n}|_{\cX_{\bar s}}\twoheadrightarrow  \cP_{m,n,\lbrace 0\rbrace}|_{\cX_{\bar s}}$ is   the maximal   quotient of $ \cP_{m,n}|_{\cX_{\bar s}}$ in $\Perv(\cX_{\bar s})$ with support in $ \lbrace 0\rbrace$.  In particular, the canonical injective  morphism
 $$\SHom_{\hbox{\rm \tiny Perv}(\cX_{\bar s})}( \cP_{m,n,\lbrace 0\rbrace}|_{\cX_{\bar s}},\delta_0)\rightarrow \SHom_{\hbox{\rm \tiny Perv}(\cX_{\bar s})}(\cP_{m,n}|_{\cX_{\bar s}},\delta_0)$$
 is an isomorphism. But as the full subcategory $\Perv_0(\cX_{\bar s})\subset \Perv(\cX_{\bar s})$ of all objects  with support in $\lbrace 0\rbrace$ identifies with $\Perv(\bar 0)\simeq \Vect_{\overline{\Q}_\ell}$ \textit{via}
 $$0_*:\Perv(\bar 0)\tilde{\rightarrow} \Perv_0(\cX_{\bar s}),$$
one has $\cP_{m,n,\lbrace 0\rbrace}|_{\cX_{\bar s}}\simeq \delta_0^{\oplus \mu_s}$ with
  $$\mu_s:=\Sdim_{\overline{\Q_\ell}}(\SHom_{\hbox{\rm \tiny Perv}(\cX_{\bar s})}(\cP_{m,n,\lbrace 0\rbrace}|_{\cX_{\bar s}},\delta_0)=\chi(\cX_{\bar s},  \cP_{m,n,\lbrace 0\rbrace}|_{\cX_{\bar s}}).$$ 
 This eventually reduces Corollary \ref{Cor:MC2Bis}   to proving that for every $b\geq 0$ the subset $$U_{\leq b} :=\lbrace s\in S\;|\; \mu_s\leq  b\rbrace\subset S$$
is open. As the  map $  \mu:S\rightarrow \Z_{\geq 0} $
  is constructible, it is enough to prove that $U_{\leq b} $ is stable under generization. This essentially follows from the existence of the cospecialization morphism since,   for every specialization $t_1\rightsquigarrow t_0$ of points in $S$,   $csp_{\bar t_1,\bar t_0} $ identifies $G(\cP_{m,n}|_{\cX_{\bar t_0}})$ with a subgroup $$G(\cP_{m,n}|_{\cX_{\bar t_0}})\subset G(\cP_{m,n}|_{\cX_{\bar t_1}})\subset \SGL(\omega_{\bar t_1}(\cP_{m,n}|_{\cX_{\bar t_1}})),$$ so that $\mu_{t_0}\geq \mu_{t_1}$. \end{proof}

  \begin{lemme}\label{Lem:TensorSS} Let $k$ be an algebraically closed field of characteristic $0$, let $X$ be an abelian variety over $k$ and let $\cP\in \Perv(X)$. Assume $\cP$ is semisimple in $ \Perv(X)$. Then for every integers $m,n\geq 0$, ${}^p\SH^0(T^{m,n}(\cP))$ is again semisimple in $\Perv(X)$.
 \end{lemme} 
 \begin{proof}  \textit{(Sketch of)} This is mentioned as \cite[Ex.~5.1]{KW}. The fact that $Rm_*:D_c^b(X\times X)\rightarrow D_c^b(X)$ preserves direct sums of shifts of simple perverse sheaves follows from Kashiwara's conjecture (Kashiwara's conjecture is reduced to a conjecture of de Jong in \cite{KashiDrinfeld}, and de Jong's conjecture is proved in \cite{deJongConj1}, \cite{deJongConj2}) while the fact that the exterior tensor product $\cP_1\boxtimes^L \cP_2$ of two  simple objects  $\cP_1,\cP_2\in \Perv(X)$ is  a  direct sums of shifts of simple perverse sheaves follows from the structure of simple perverse sheaves and  the fact that for every   immersion $\iota_i:U_i\hookrightarrow X$
 and $\cL_i\in \Loc(U_i)$, 
 $i=1,2$ one has $$(\iota_1\times \iota_2)_{!*}( \cL_1\boxtimes^L    \cL_2) \simeq  \iota_{1,!*}(\cL_1)\boxtimes^L  \iota_{2,!*}(\cL_2).$$
 See \textit{e.g.} \cite[Ex. 10.2.31]{MS}.
 \end{proof}

 \begin{lemme}\label{MaximalSupport} Let     $f:\cX\rightarrow S$ be a separated morphism of finite type and let  $\iota:\cZ\hookrightarrow \cX$ be a closed immersion.  For every $\cP\in \Perv(\cX/S)$, there is a  quotient $\cP\twoheadrightarrow\cP_\cZ $
 in $\Perv(\cX/S)$ such that   for every geometric point $\bar s$ on $ S$, $ \cP|_{\cX_{\bar s}}\twoheadrightarrow  \cP_\cZ|_{\cX_{\bar s}}$ is   the maximal   quotient of $ \cP|_{\cX_{\bar s}}$ in $\Perv(\cX_{\bar s})$ with support in $\cZ_{\bar s}$.  
 \end{lemme}
 \begin{proof}  Define   $\cP\twoheadrightarrow\cP_\cZ$ as the image of the composite
 $$\cP\stackrel{\hbox{\rm\tiny ad}_\iota}{\rightarrow} \iota_*\iota^*\cP \rightarrow  {}^{p/S}\tau^{\geq 0}(\iota_*\iota^* \cP)$$
 of the adjunction morphism for $\iota:\cZ\hookrightarrow \cX$ and the relative perverse truncation with respect to $f:\cX\rightarrow S$. 
  As $\iota_*:D_c^b(\cZ)\rightarrow D_c^b(\cX)$ is $t$-exact and  $\iota^*:D_c^b(\cX)\rightarrow D_c^b(\cZ)$ is right $t$-exact with respect to the relative perverse $t$-structure on $f:\cX\rightarrow S$, one has 
 $${}^{p/S}\tau^{\geq 0}(\iota_*\iota^* -)\simeq \iota_*{}^{p/S}\tau^{\geq 0}(\iota^*-)\simeq \iota_*{}^{p/S}\SH^0(\iota^*-):\Perv(\cX/S)\rightarrow \Perv(\cX/S).$$
 Furthermore,  for every geometric point $\bar s$ on $ S$, by proper base-change,
  $$ -|_{\cX_{\bar s}}\circ \iota_* \simeq \iota_{\bar s *} \circ -|_{\cZ_{\bar s}} :\Perv(\cZ/S)\rightarrow \Perv(\cX_{\bar s}),$$
  while, by definition of the relative perverse $t$-structure,
  $$-|_{\cZ_{\bar s}} \circ {}^{p/S}\SH^0(-) \simeq  {}^p\SH^0(-|_{\cZ_{\bar s}} )  :\Perv(\cZ/S)\rightarrow \Perv(\cZ_{\bar s}).$$
  This proves that for every geometric point $\bar s$, the formation of $\cP\twoheadrightarrow\cP_\cZ$ commutes with $-|_{\cX_{\bar s}}:\Perv(\cX/S)\rightarrow \Perv(\cX_{\bar s}) $ and reduces the proof of Lemma \ref{MaximalSupport} to the case where $S=\spec(\bar k)$ is the spectrum of an algebraically closed field. By construction $\cP_\cZ$ has support in $\cZ$. Let $j:\cX\setminus \cZ\hookrightarrow \cX$ denote the complementary open immersion. Conversely, for every quotient $\cP\twoheadrightarrow \cQ$ in $\Perv(\cX)$ with support in $\cZ$, the distinguished triangle 
  $$j_!j^*\cQ\rightarrow \cQ\rightarrow \iota_*\iota^*\cQ\stackrel{+1}{\rightarrow}$$
  ensures that $\cQ\tilde{\rightarrow} \iota_*\iota^*\cQ$ so that, by adjunction, $\cP\twoheadrightarrow \cQ\tilde{\rightarrow} \iota_*\iota^*\cQ$ factors as 
\begin{equation}\label{Diag:MaximalSupport}\xymatrix{\cP\ar@{->>}[rr]\ar[dr]_{\hbox{\rm\tiny ad}_\iota} && \iota_*\iota^*\cQ\\
  & \iota_*\iota^*\cP\ar@{.>}[ur] &},\end{equation}
  and, as $\cQ\in \Perv(\cX)$, (\ref{Diag:MaximalSupport}) factors further as 
  $$\xymatrix{\cP\ar@{->>}[rr]\ar[dr]_{\hbox{\rm\tiny ad}_\iota} && \iota_*\iota^*\cQ\\
  & \iota_*\iota^*\cP\ar[ur] \ar[r]& {}^{p/S}\tau^{\geq 0}(\iota_*\iota^* \cP),\ar@{.>}[u]}$$
  which concludes the proof of Lemma \ref{MaximalSupport}. \end{proof}

   \subsection{Proof of Corollary \ref{Cor:MTBis}}  Let   $f:\cX\rightarrow S$ be an abelian scheme. We begin with the following observation.

 \begin{lemme}\label{Lem:Det} Let  $\cP\in \Perv^{\ULA}(\cX/S)$. Assume $\cP|_{\cX_{\bar \eta}}$ has torsion determinant of order $N$. Then  for every $s\in S$, $\cP|_{\cX_{\bar s}}$ also has torsion determinant of order dividing $N$. 
 \end{lemme}

\begin{proof}  Let $n$ denote the dimension of $\cP$ in $P^{\ULA}(\cX/S)$ and $\det(\cP):=\wedge^n\cP$ its determinant. It follows from the general formalism of Tannakian categories that for every $t\in S$, $\cP|_{\cX_t}$, $\cP|_{\cX_{\bar t}}$ again have dimension $n$ and that $(\wedge^n\cP)|_{\cX_t}\simeq \wedge^n(\cP|_{\cX_t})$, $(\wedge^n\cP)|_{\cX_{\bar t}}\simeq \wedge^n(\cP|_{\cX_{\bar t}})$. In particular, as for every $s\in S$, one has 
 $$G( \wedge^n(\cP|_{\cX_{\bar s}}))\subset G( \wedge^n(\cP|_{\cX_{\bar \eta}})),$$ 
hence $\wedge^n(\cP|_{\cX_{\bar s}})$ is also torsion, with order    dividing $N$. \end{proof}

\noindent Note that if $G(\cP|_{\cX_{\bar \eta}})$ is semisimple then any object in $\langle \cP|_{\cX_{\bar \eta}}\rangle$ has torsion determinant of order dividing $|\pi_0(G(\cP|_{\cX_{\bar \eta}}))|$.\\

   \noindent  We now turn to the proof of Corollary \ref{Cor:MTBis} itself. Let $\cP\in \Perv^{\ULA}(\cX/S)$. After possibly replacing $\cP\in  \Perv^\ULA(\cX/S)$ with $\cP|_{\cX_{\bar k}}\in   \Perv^\ULA(\cX_{\bar k}/S_{\bar k})$, one may assume  $k=\bar k$ is algebraically closed.  \\

    \begin{itemize}[leftmargin=*, parsep=0cm,itemsep=0.2cm,topsep=0.2cm]
 \item  Proof of Corollary \ref{Cor:MTBis} (1).  Up to replacing $\cP$ with  $\cP^*$ - see Corollary \ref{Cor:LiftingSS}, one may assume $\cP|_{\cX_{\bar \eta}}$ is simple in $\Perv(\cX_{\bar \eta})$, and not only in $P(\cX_{\bar \eta})$. Then 
   Corollary \ref{Cor:MTBis} (1) immediately follows from  Fact \ref{Fact:KramerDiss},   Theorem \ref{MT} and  Lemma \ref{Lem:Det}.

 \item  Proof of Corollary \ref{Cor:MTBis} (2).  Let us first observe that for a connected reductive group $G$ and a closed subgroup $H\subset G$, the following are equivalent
    \begin{enumerate}[leftmargin=*, parsep=0cm,itemsep=0.2cm,topsep=0.2cm]
 \item $H\subset G\twoheadrightarrow G^{\semis}$ factors through an isogeny $H^{\semis}\twoheadrightarrow G^{\semis}$;
 \item dim$(R(G)\cap H)=$dim$(R(H))$ and dim$(H)-$dim$(R(H))=$dim$(G)-$dim$(R(G))$;
  \item dim$(R(G)\cap H^\circ)=$dim$(R(H^\circ))$ and dim$(H^\circ)-$dim$(R(H^\circ))=$dim$(G)-$dim$(R(G))$;
  \item  $H^\circ\subset G\twoheadrightarrow G^{\semis}$ factors through an isogeny $H^{\circ,\semis}\twoheadrightarrow G^{\semis}$.
 \end{enumerate}
 In particular, to prove the first part of  Corollary \ref{Cor:MTBis} (2), one can replace $G(\cP|_{\cX_{\bar s}})^\circ$ with $G(\cP|_{\cX_{\bar s}})\cap G(\cP|_{\cX_{\bar \eta}})^\circ$ (which will simplify a bit the notation).\\
 
 \noindent   Replacing $\cP\in \Perv^{\ULA}(\cX/S)$ with $[N]_*\cP\in \Perv^{\ULA}(\cX/S)$ for some integer $N\geq 1$, one may assume $G(\cP|_{\cX_{\bar \eta}})$ is connected (see \cite[\S 2]{W}). From the short exact sequence 
   $$1\rightarrow G(\cP|_{\cX_{\bar \eta}})\rightarrow G(\cP )\rightarrow G( \langle\cP \rangle_0 )\rightarrow 1$$
 and the description of $G( \langle\cP \rangle_0 )$ in terms of representation of $\pi_1(S)$ (see Subsection \ref{Sec:ProofHilbert} below), 
  replacing $\cP\in \Perv^{\ULA}(\cX/S)$ with $\cP|_{\cX_{S'}}\in \Perv^{\ULA}(\cX_{S'}/S')$ for some connected \'etale cover $S'\rightarrow S$, one may also assume $G(\cP)$ is connected. \\

 \noindent   Let $$G(\cP)\twoheadrightarrow G(\cP)^{\ad}$$
  denote the maximal \textit{adjoint} quotient\footnote{Namely, $G(\cP)^{\ad}=G(\cP)^{\red}/Z(G(\cP)^{\red})$, where $G(\cP)\twoheadrightarrow G(\cP)^{\red}:=G(\cP)/R_u(G(\cP))$ is the maximal reductive quotient of $G(\cP)$.} of $G(\cP)$. The morphism $G(\cP)\twoheadrightarrow G(\cP)^{\ad}$ factors as $G(\cP)\twoheadrightarrow G(\cP)^{\semis}\twoheadrightarrow G(\cP)^{\ad}$. On the other hand, as $G(\cP|_{\cX_{\bar \eta}})$ is normal in $ G(\cP)$, $R(G(\cP|_{\cX_{\bar \eta}}))=(R(G(\cP))\cap G(\cP|_{\cX_{\bar \eta}}))^\circ$ so that the  morphism $G(\cP|_{\cX_{\bar \eta}}) \hookrightarrow G(\cP)\twoheadrightarrow G(\cP)^{\semis}$ factors through a  morphism   $G(\cP|_{\cX_{\bar \eta}})^{\semis} \rightarrow  G(\cP)^{\semis}$ inducing an isogeny onto its image, 
 which is  a closed normal subgroup of  $G(\cP)^{\semis}$. By the structure theory of connected semisimple groups, there is a (unique) connected (automatically adjoint) quotient $G(\cP)^{\ad}\twoheadrightarrow \widetilde{G}$ such that the resulting canonical morphism
   $$G(\cP|_{\cX_{\bar \eta}})^{\semis}\rightarrow  G(\cP)^{\semis}\twoheadrightarrow G(\cP)^{\ad}\twoheadrightarrow \widetilde{G}$$
   is an isogeny. As $ \widetilde{G}$ is adjoint, it admits an irreducible faithful representation corresponding to a
simple object  $\cQ\in \langle \cP\rangle$; in particular,   $\widetilde{G}=G(\cQ)$. The commutative diagram of exact tensor functors
\begin{equation}\label{Diag:Cat}\xymatrix{\langle \cQ|_{\cX_{\bar \eta}}\rangle\ar@{_{(}->}[d]\ar@/^1.5pc/[rrr]^{  csp_{\bar \eta,\bar s}}&\langle \cQ\rangle\ar[l]_{|_{\cX_{\bar \eta}}}\ar[r]^{|_{\cX_s}}\ar@{_{(}->}[d]&\langle \cQ|_{\cX_s}\rangle\ar[r]^{|_{\cX_{\bar s}}}\ar@{_{(}->}[d]&\langle \cQ|_{\cX_{\bar s}}\rangle\ar@{_{(}->}[d]\\
\langle \cP|_{\cX_{\bar \eta}}\rangle\ar@/_1.5pc/[rrr]_{ csp_{\bar \eta,\bar s}} &\langle \cP\rangle\ar[l]_{|_{\cX_{\bar \eta}}}\ar[r]^{|_{\cX_s}} &\langle \cP|_{\cX_s}\rangle\ar[r]^{|_{\cX_{\bar s}}} &\langle \cP|_{\cX_{\bar s}}\rangle \\}\end{equation}
induces a commutative diagram of algebraic groups 
\begin{equation}\label{Diag:Group}\xymatrix{G(\cQ|_{\cX_{\bar \eta}})\ar@{^{(}->}[r]&G (\cQ)  &G( \cQ|_{\cX_s}) \ar@{_{(}->}[l]&G( \cQ|_{\cX_{\bar s}})  \ar@{_{(}->}[l]\ar@/_1.5pc/[lll]_{ csp_{\bar \eta,\bar s}}\\
G(\cP|_{\cX_{\bar \eta}})^{\semis}\ar@{->>}[u] \ar[r]&G( \cP)^{\semis}  \ar@{->>}[u]& &   \\
G(\cP|_{\cX_{\bar \eta}})\ar@{->>}[u] \ar@{^{(}->}[r]&G( \cP)  \ar@{->>}[u]&G( \cP|_{\cX_s})  \ar@{_{(}->}[l] \ar@{->>}[uu]&G(\cP|_{\cX_{\bar s}}) \ar@{_{(}->}[l]\ar@{->>}[uu]\ar@/^1.5pc/[lll]^{ csp_{\bar \eta,\bar s}}  \\}\end{equation}
As $G(\cP|_{\cX_{\bar \eta}})^{\semis}\twoheadrightarrow G (\cQ) $ is an isogeny, $G(\cQ|_{\cX_{\bar \eta}})\rightarrow G (\cQ) $ is an isomorphism. In particular, $\cQ|_{\cX_{\bar \eta}}$ is a simple object in $P(\cX_{\bar \eta})$ and every object in $\langle \cQ|_{\cX_{\bar \eta}}\rangle$ has trivial determinant. By Corollary  \ref{Cor:MTBis} (1) applied to $\cQ$ up to replacing $S$ by a non-empty open subscheme, one may assume that for every $s\in S$ the cospecialization morphism $ csp_{\bar \eta,\bar s}:G( \cQ|_{\cX_{\bar s}})\hookrightarrow G( \cQ|_{\cX_{\bar \eta}})$ is an isomorphism so that  one has a canonical commutative diagram
\begin{equation}\label{Diag:GroupSpec}\xymatrix{G(\cQ|_{\cX_{\bar \eta}}) &  &G( \cQ|_{\cX_{\bar s}})  \ar[ll]_{ csp_{\bar \eta,\bar s}}^\simeq\\
G(\cP|_{\cX_{\bar \eta}})^{\semis}\ar@{->>}[u]   &G(\cP|_{\cX_{\bar s}}) /(R(G(\cP|_{\cX_{\bar \eta}}))\cap G(\cP|_{\cX_{\bar s}})) \ar@{_{(}->}[l]  & G(\cP|_{\cX_{\bar s}})^{\semis} \ar@{->>}[u]\\
&G(\cP|_{\cX_{\bar s}}) /(R(G(\cP|_{\cX_{\bar \eta}}))\cap G(\cP|_{\cX_{\bar s}}))^\circ \ar@{->>}[u] \ar@{->>}[ur]& \\
G(\cP|_{\cX_{\bar \eta}})\ar@{->>}[uu]   && G(\cP|_{\cX_{\bar s}}). \ar@{->>}[uu]\ar@{_{(}->}[ll]^{ csp_{\bar \eta,\bar s}}\ar@{->>}[ul]\ar@{->>}[uu]  \\}\end{equation}
In particular, $$\begin{tabular}[t]{ll}
$\Sdim(G(\cP|_{\cX_{\bar s}})^{\semis})\geq \Sdim(G(\cP|_{\cX_{\bar \eta}})^{\semis})$&$\geq \Sdim(G(\cP|_{\cX_{\bar s}}) /(R(G(\cP|_{\cX_{\bar \eta}}))\cap G(\cP|_{\cX_{\bar s}})))$\\
&$= \Sdim(G(\cP|_{\cX_{\bar s}}) /(R(G(\cP|_{\cX_{\bar \eta}}))\cap G(\cP|_{\cX_{\bar s}}))^\circ)$\\
&$\geq \Sdim(G(\cP|_{\cX_{\bar s}})^{\semis})$
\end{tabular}$$
which, as $G(\cP|_{\cX_{\bar \eta}})$ - hence $G(\cP|_{\cX_{\bar \eta}})^{\semis}$ -  are connected, imposes  that the  morphisms 
   $$ 
G(\cP|_{\cX_{\bar s}}) /(R(G(\cP|_{\cX_{\bar \eta}}))\cap G(\cP|_{\cX_{\bar s}}))   \stackrel{\simeq}{\rightarrow}G(\cP|_{\cX_{\bar \eta}})^{\semis} $$
and 
$$G(\cP|_{\cX_{\bar s}}) /(R(G(\cP|_{\cX_{\bar \eta}}))\cap G(\cP|_{\cX_{\bar s}}))^\circ  \stackrel{\simeq}{\rightarrow}   G(\cP|_{\cX_{\bar s}})^{\semis},$$
 are isomorphisms.
This concludes
the proof of the first part of   Corollary  \ref{Cor:MTBis} (2). The second part  when $G(\cP|_{\cX_{\bar \eta}})$ is semisimple tautologically follows from the first part as, then, $G(\cP|_{\cX_{\bar \eta}})^\circ=G(\cP|_{\cX_{\bar \eta}})^{\circ,\semis}$ while the second  part  when $G(\cP|_{\cX_{\bar \eta}})$ is reductive  follows from the first part  and the fact that, for every $s\in S $ such that $G(\cP|_{\cX_{\bar s}})^\circ\subset G(\cP|_{\cX_{\bar \eta}})^\circ $ factors through an isogeny $G(\cP|_{\cX_{\bar s}})^{\circ,\semis} \twoheadrightarrow  G(\cP|_{\cX_{\bar \eta}})^{\circ,\semis} $, the arrows $(*\hbox{\rm-}\bar s)$, $(*\hbox{\rm-}\bar \eta)$ and the right vertical arrow in the canonical commutative diagram
  $$\xymatrix{G(\cP|_{\cX_{\bar s}} )^{\circ,\der}\ar@{^{(}->}[r]\ar@{^{(}->}[d]\ar@/^1.5pc/[rr]^{(*\hbox{\rm-}\bar s)}&G(\cP|_{\cX_{\bar s}} )^\circ\ar@{->>}[r]\ar@{^{(}->}[d]&G(\cP|_{\cX_{\bar s}} )^{\circ,\semis}\ar[d] \\
G(\cP|_{\cX_{\bar \eta}} )^{\circ,\der}\ar@{^{(}->}[r] \ar@/_1.5pc/[rr]_{(*\hbox{\rm-}\bar \eta)}&G(\cP|_{\cX_{\bar\eta}} )^\circ\ar@{->>}[r]&G(\cP|_{\cX_{\bar \eta}} )^{\circ,\semis}.}$$
 
are isogenies.
\end{itemize}

 \subsection{Proof of  Proposition \ref{Prop:Hilbert}}\label{Sec:ProofHilbert} Let $k$ be a field of characteristic $0$, $S $ a  smooth  geometrically connected variety over $k$  with generic point $\eta$ and   $f:\cX\rightarrow S$ an abelian scheme.   Let $P^\ULA(\cX/S)_0\subset P^\ULA(\cX/S)$ denote the essential image of 
 $$0_*:\Loc(S)=\Perv^\ULA(S/S)\rightarrow\Perv^\ULA(\cX/S) \rightarrow P^\ULA(\cX/S)$$
 and for every $\cP\in P^\ULA(\cX/S)$, consider the full subcategory $\langle \cP\rangle_0:=\langle \cP\rangle\cap P^\ULA(\cX/S)_0\subset \langle \cP\rangle$. 
 From Lemma \ref{Lem:SpecAr}, the equivalence of Tannakian categories $-|_{\cX_\eta}:\langle \cP\rangle\tilde{\rightarrow}\langle \cP|_{\cX_\eta}\rangle$ restricts to an equivalence 
 
 $$-|_{\cX_\eta}:\langle \cP\rangle_0\tilde{\rightarrow}\langle \cP|_{\cX_\eta}\rangle_0.$$
 This yields an explicit categorical  description of the morphism $G(\langle \cP|_{\cX_s}\rangle_0)\rightarrow   G(\langle \cP|_{\cX_\eta}\rangle_0)$ as the composite $G(\langle \cP|_{\cX_s}\rangle_0)\rightarrow   G(\langle \cP \rangle_0)\tilde{\leftarrow}   G(\langle \cP|_{\cX_\eta}\rangle_0)$ arising from the diagram of Tannakian categories
 $$\langle \cP|_{\cX_s}\rangle_0\stackrel{|_{\cX_s}}{\leftarrow}\langle \cP \rangle_0\stackrel{|_{\cX_\eta}}{\rightarrow}\langle \cP|_{\cX_\eta}\rangle_0$$

\noindent Assume furthermore  $S_\cP^{\geo}=\emptyset$ that is, for every $s\in S$, the cospecialization  morphism $G(\cP|_{\cX_{\bar s}} )\rightarrow G(\cP|_{\cX_{\bar \eta}} )$ is an isomorphism so that  the morphism 
 $G( \langle\cP|_{\cX_s}\rangle_0 )\hookrightarrow G(\langle \cP \rangle_0)$
 is a closed immersion. Then every $\otimes$-generator $0_*\cL$ of $\langle \cP \rangle_0$ yields a $\otimes$-generator $(0_*\cL)|_{\cX_s}\simeq 0_{s *}s^*\cL$ of $\langle\cP|_{\cX_s}\rangle_0 $ and from the canonical diagram of Tannakian categories

 $$\xymatrix{\langle \cL\rangle\ar@{^{(}->}[r]\ar[d]_\simeq^{0_*}\ar@/^2pc/[rrr]&\Perv^\ULA(S/S)\ar[r]^{s^*}\ar[d]_\simeq^{0_*}&\Perv(s)\ar[d]_\simeq^{0_{s*}}&\langle s^*\cL\rangle\ar@{_{(}->}[l]\ar[d]_\simeq^{0_{s*}}\\
\langle 0_*\cL\rangle\ar@{^{(}->}[r] \ar@/_2pc/[rrr]&P^\ULA(\cX/S)_0\ar[r]^{|_{\cX_s}} &P(\cX_s)_0&\langle 0_{s*}s^*\cL \rangle\simeq\langle  (0_*\cL)|_{\cX_s} \rangle\ar@{_{(}->}[l] }$$
 \medskip
 
\noindent  the morphism 
 $G( \langle\cP|_{\cX_s}\rangle_0 )\hookrightarrow G(\langle \cP \rangle_0)$ also describes the functor of Tannakian categories $s^*:\langle \cL\rangle\rightarrow \langle s^*\cL\rangle$ hence corresponds to the embedding $$G(\cL)_s\hookrightarrow G(\cL)\hookrightarrow \SGL(\cL_{\bar s})$$
 of the Zariski-closures of  the images $$\Pi(\cL)_s\subset \Pi(\cL)\subset \SGL(\cL_{\bar s})$$ of $\pi_1(s,\bar s)\rightarrow \pi_1(S,\bar s)$ acting on $\cL_{\bar s}$ respectively. \\

\noindent This observation yields the following.

\begin{lemme}\label{MC3Bis} Assume $S$ has dimension $>0$. Let $\cP\in\Perv^\ULA(\cX/S)$ with $S_{\cP}^{\geo}=\emptyset $. Then, 
  \begin{enumerate}[leftmargin=*, parsep=0cm,itemsep=0.2cm,topsep=0.2cm]
  \item if  $k$ is Hilbertian,   there exists an integer $d\geq 1$ such that $|S|^{\leq d}\setminus S_{\cP}\cap |S|^{\leq d}$ is infinite.
  \item if $S$ is a curve, $k$ is finitely generated over $\Q$ and $G(\cP|_{\cX_{\bar k}})$ is semisimple,  for every integer $d\geq 1$,  $ S^\circ_{\cP}\cap |S|^{\leq d}$ is finite.
\end{enumerate}
 \end{lemme}

 \begin{proof} From the exact  specialization diagram  (\ref{Diag:Rep}) and the fact that, by our assumptions, for every $s\in |S|$ the morphisms $G(\cP|_{\cX_{\bar s}}) \rightarrow G(\cP|_{\cX_{\bar \eta}} )$  is an isomorphism and the morphism  $G( \langle\cP|_{\cX_s}\rangle_0 )\rightarrow G(\langle \cP \rangle_0)$ is a closed immersion, it is enough to prove that, under the assumptions
  \begin{itemize}[leftmargin=*, parsep=0cm,itemsep=0.2cm,topsep=0.2cm]
  \item   in (1): there exists an integer $d\geq 1$ such that for infinitely many  $s\in |S|^{\leq d} $  the closed immersion $G( \langle\cP|_{\cX_s}\rangle_0 )\hookrightarrow G(\langle \cP \rangle_0)$ is an isomorphism. \\
  
  \noindent   This  follows from the defining property of Hilbertian fields and a Frattini argument \cite[\S 10.6]{LMW}, which ensures  that there exists an integer $d\geq 1$ such that for infinitely many  $s\in |S|^{\leq d} $ one has $\Pi(\cL)_s= \Pi(\cL)$.

  \item in (2): for every integer $d\geq 1$ and for all but finitely many  $s\in |S|^{\leq d} $  the closed immersion $G( \langle\cP|_{\cX_s}\rangle_0 )^\circ\hookrightarrow G(\langle \cP \rangle_0)^\circ$ is an isomorphism. \\

  \noindent   This  follows from \cite[Thm.~1]{UOI2}, which asserts that if $\rho: \pi_1(S)\rightarrow \SGL_N(\Z_\ell)$ is a continuous GLP representation then, for every integer $d\geq 1$ and all but finitely many $s\in |S|^{\leq d}$, $\rho(\pi_1(s))\subset \rho(\pi_1(S))$ is open. The GLP condition means that every open subgroup of $ \Pi:=\rho(\pi_1(S_{\bar k}))$ has finite abelianization or, equivalently, that the Lie algebra $\SLie(\Pi)$ of $\Pi$ (as an $\ell$-adic Lie group) is perfect. This is for instance the case if (*)  one can realize $\Pi$ as a closed subgroup $\Pi\subset H_0(\Q_\ell)$ of the group of $\Q_\ell$-points of an algebraic group $H_0$ over $\Q_\ell$ such that the Zariski-closure of $\Pi$ in $H_0$ is  semisimple.  The assumption that $G(\cP|_{\cX_{\bar k}})$ is semisimple ensures that one can reduce to this situation. Indeed, as   $G(\cL|_{S_{\bar k}})$ is a quotient of $G(\cP|_{\cX_{\bar k}})$, $G(\cL|_{S_{\bar k}})$ is semisimple as well and, as there exists a finite Galois   extension $Q_\ell$ of $\Q_\ell$ such that $\cL$ arises from a $Q_\ell$-local system on $S$, one may assume $\Pi(\cL)\subset \SGL(\cL_{\bar s})\simeq \SGL_r(Q_\ell)$. But as   $\SGL_r(Q_\ell)$ has a natural structure of  Lie group over $\Q_\ell$, so has  $\Pi(\cL) $   \cite[L.G., Chap. V, \S 9]{LGLA} hence, as $\Pi(\cL) $ is also compact being the continuous image of a profinite group, it admits a faithful   embedding into $\SGL_N(\Z_\ell)$ for some $N\geq 1$ \cite[Prop. 4]{Lubotsky}. To apply  \cite[Thm. 1]{UOI2} to the resulting $\ell$-adic representation $ \pi_1(S)\twoheadrightarrow \Pi(\cL)\subset \SGL_N(\Z_\ell)$, it is thus enough to show that $ \Pi:=\Pi(\cL|_{S_{\bar k}})$ satisfies the criterion (*). This follows from the claim below, applied with $K/k=Q_\ell/\Q_\ell$, $\Pi:=\Pi(\cL|_{S_{\bar k}})$ and $G_0:=\SGL_{r,\Q_\ell}$.
  \end{itemize}

 \noindent\textbf{Claim.} \textit{Let $K/k$ be a finite Galois extension and write $R:=\hbox{\rm Res}_{K|k}:Sch/_K\rightarrow Sch/_k$ for the Weil restriction functor. Let  $G_0$ be an algebraic group over $k$ and set $G:=G_{0,K}$. Let  $\Pi\subset G(K)=(RG)(k)$ be a subgroup. Let $\iota: H\hookrightarrow G$ denote the Zariski closure of $\Pi$ in $G $ and $\iota_0:H_0\hookrightarrow RG $ the Zariski-closure of $\Pi$ in $RG$. Write $ad:G_0\hookrightarrow RG$ for the adjunction morphism. Then the   morphism $c:H\stackrel{\iota}{\hookrightarrow}  G\stackrel{ad_K}{\hookrightarrow} (RG)_K$ factors through an isomorphism} 
 $$\xymatrix{H\ar@{^{(}->}[r]^\iota \ar@{.>}[drr]^\simeq_c &G\ar@{^{(}->}[r]^{ad_K} &(RG)_K\\
 &&H_{0,K}.\ar@{_{(}->}[u]_{\iota_{0,K}}}$$
 \noindent\textit{Proof of the claim.} At the level of $K$-points, the diagram $$\xymatrix{H\ar@{^{(}->}[r]^\iota   &G\ar@{^{(}->}[r]^{ad_K} &(RG)_K
 & H_{0,K}\ar@{_{(}->}[l]_{\iota_{0,K}}}$$
 induces a commutative diagram
 $$\xymatrix{H(K)\ar@{^{(}->}[r]   &G(K)\ar@{^{(}->}[r] &(RG)_K(K)
 & H_{0,K}(K)\ar@{_{(}->}[l] \\
 &\Pi\ar@{_{(}->}[ul]\ar@/_0.5cm/@{^{(}->}[rr]&RG(k)\ar@{_{(}->}[u]\ar@{=}[ul]&H_0(k).\ar@{_{(}->}[u]\ar@{_{(}->}[l] }$$
As $\Pi$ is Zariski-dense in $H$, this already shows the existence of the factorization $c:H\hookrightarrow H_{0,K}$. On the other hand, at the level of $k$-points the diagram 
$$\xymatrix{RH\ar@{^{(}->}[r]^{R\iota}  &RG &H_0 \ar@{_{(}->}[l]_{\iota_0} }$$
 induces a commutative diagram
$$\xymatrix{RH(k)=H(K)\ar@{^{(}->}[r]  &(RG)(k)=G(K) \\
\Pi \ar@{^{(}->}[r]\ar@{_{(}->}[u]&H_0(k)\ar@{_{(}->}[u] }$$
As $\Pi$ is Zariski-dense in $H_0$, this shows that $H_0 \stackrel{\iota_0}{\hookrightarrow} RG$ factors as $\iota_0: H_0\stackrel{d_0}{\hookrightarrow} RH\stackrel{R\iota}{\hookrightarrow} RG$.  One thus gets 
 $$\xymatrix{RH\ar@{^{(}->}[r]^{R\iota} \ar@/_1cm/@{^{(}->}[drr]_{Rc} &RG\ar@{^{(}->}[r]^{R(ad_K)} &R((RG)_K)\\
 &H_0\ar@{^{(}->}[r]^{ad}\ar@{_{(}->}[u]_{\iota_0}\ar@{_{(}->}[ul]_{d_0}&R(H_{0,K})\ar@{_{(}->}[u]_{R(\iota_{0,K})}}$$
Let $d:H_{0,K}\rightarrow H$ denote the morphism corresponding by functoriality, to $d_0:H_0\hookrightarrow RH$. Then, by construction,  $c\circ d=Id: H_{0,K}\tilde{\rightarrow} H_{0,K}$ and $d\circ c=Id:H\tilde{\rightarrow} H $.  \end{proof}

\noindent  Proposition \ref{Prop:Hilbert} (and its strengthening when $S$ is a curve and $k$ is finitely generated over $\Q$) follows from   Lemma \ref{MC3Bis} applied to the restriction of $\cP$ to $\cX\times_SU$, where $U\subset S$ denotes the complement of the Zariski-closure of $S_{\cP}^{\geo}$ in $S$. 
 
 \section{Geometric applications}\label{Sec:GeoApp}\textit{}\\
 \noindent Let $\cX\rightarrow S$ be an abelian scheme  and let $ \cY\hookrightarrow \cX $ be a closed subscheme, smooth and  geometrically connected over $S$. \\
 \subsection{Preliminaries}\label{Par:Stab} As $S$ is smooth, $ \cX\rightarrow S$ is projective \cite[Thm. XI.1.4]{MichelAmple} hence  $\cX$ carries a line bundle $\cO_{\cX}(1)$ which is very ample with respect to $\cX\rightarrow S$. Let $P\in \Q[T]$ denote the Hilbert polynomial of $ \cY_{\bar \eta}\hookrightarrow \cX_{\bar \eta}$ with respect to $\cO_{\cX}(1)|_{\cX_{\bar \eta}}$ and let $\frak{H}_{\cX/S}^P\rightarrow S$ be the Hilbert scheme  classifying closed subschemes of $\cX\times_ST$ which are flat over $T$ and with constant Hilbert polynomial $P$  \cite[Thm. 3.2]{GrothendieckHilbert}. By construction, $\cX $ acts   by translation on  $\frak{H}_{\cX/S}^P$ over $S$. Let $\hbox{\rm \small $[\cY]$}\in \frak{H}_{\cX/S}^P(S)$ be the $S$-point corresponding to $\iota:\cY\hookrightarrow \cX$ and consider the corresponding morphism of $S$-schemes $$\phi_{\hbox{\rm \tiny $[\cY]$}}:\cX\rightarrow \frak{H}_{\cX/S}^P\times_S\frak{H}_{\cX/S}^P,\;\;  x\mapsto (\hbox{\rm \small $[\cY]$},\hbox{\rm \small $[\cY$}+x\hbox{\rm \small $]$}).$$
Let also $$\Delta:\frak{H}_{\cX/S}^P\hookrightarrow \frak{H}_{\cX/S}^P\times_S\frak{H}_{\cX/S}^P$$
denote the diagonal embedding, which is a closed immersion as $\frak{H}_{\cX/S}^P\rightarrow S$ is projective. Define the stabilizer $\hbox{\rm Stab}_{\cX/S}(\cY)$ of $\cY$ in $\cX$ as the fiber product  
$$\xymatrix{\hbox{\rm Stab}_{\cX/S}(\cY)\ar@{^{(}->}[r]\ar[d]\ar@{}[dr]|{\square}& \cX\ar[d]^{\phi_{\hbox{\rm \tiny $[\cY]$}}}\\
\frak{H}_{\cX/S}^P\ar@{^{(}->}[r]_\Delta& \frak{H}_{\cX/S}^P\times_S\frak{H}_{\cX/S}^P.}$$
By construction $\hbox{\rm Stab}_{\cX/S}(\cY)\hookrightarrow \cX$ is a closed subgroup scheme of $\cX\rightarrow S$, whose formation commutes with arbitrary Noetherian base change $T\rightarrow S$. In particular, for every $t\in S$ one has 

\begin{equation}\label{Diag:Stab}\hbox{\rm Stab}_{\cX/S}(\cY)_{\bar t}=\hbox{\rm Stab}_{\cX_{\bar t}}(\cY_{\bar t}).\end{equation}

  \begin{lemme}\label{Lem:Geo} Let $ \cX\rightarrow S$ an abelian scheme and $ \cY\hookrightarrow \cX $ a closed subscheme, smooth, geometrically connected, and of relative dimension $d $ over $S$.  Then, 
  \begin{enumerate}[leftmargin=*, parsep=0cm,itemsep=0.2cm,topsep=0.2cm]
  \item the relative perverse sheaf $\cP:=\iota_*\overline{\Q}_{\ell,Y}[d]$ lies in $ \Perv^{\ULA}(\cX/S)$;
  \item Assume  furthermore $ \cY_{\bar \eta}\hookrightarrow \cX_{\bar \eta}$ has 
    \begin{enumerate}[leftmargin=*, parsep=0cm,itemsep=0.1cm,topsep=0.1cm, label=\roman*)]
  \item  ample normal bundle $\cN_{ \cY_{\bar \eta}/\cX_{\bar \eta}}$ then, after possibly replacing $S$ by a non-empty open subset, one may assume that  for all  $s\in S$,  $\cY_{\bar s}\hookrightarrow \cX_{\bar s}$ also has    ample normal bundle $\cN_{ \cY_{\bar s}/\cX_{\bar s}}$
  \item   trivial stabilizer $\hbox{\rm Stab}_{\cX_{\bar \eta}}(\cY_{\bar \eta})$ then, after possibly replacing $S$ by a non-empty open subset, one may assume that  for all  $s\in S$,  $\cY_{\bar s}\hookrightarrow \cX_{\bar s}$ also has  trivial stabilizer.
   \end{enumerate}
  \end{enumerate}
 \end{lemme} 
\begin{proof}  For (1), as $\cY\rightarrow S$ is smooth, $\overline{\Q}_{\ell,\cY}\in D^{\ULA}(\cY/S)$ \cite[3.6, Lemma (i)]{Barrett} and as $\cY\hookrightarrow \cX$ is proper, $\iota_*\overline{\Q}_{\ell,\cY}\in D^{\ULA}(\cX/S)$  \cite[3.6, Lemma (ii)]{Barrett}  hence   $\cP:=\iota_*\overline{\Q}_{\ell,\cY}[d]\in \Perv^{\ULA}(\cX/S)$.  Assertion (2)  ii) follows from   (\ref{Diag:Stab})  and \cite[Thm. 8.10.5. (i)]{GrothendieckEGA4}. For assertion (2) i), under our assumptions for every $t\in S$ and with  the notation in the base-change diagram
$$ \xymatrix{\cY_{\bar t}\ar@{^{(}->}[r]^{\iota_{\bar t}}\ar@{}[dr]|{\square}\ar[d]\ar@/_1cm/[dd]_{\iota_{\cY_{\bar t}}}& \cX_{\bar t}\ar@{}[dr]|\square\ar[r] \ar[d] &\spec(k(\bar t))\ar[d]\\
\cY_t\ar@{^{(}->}[r]^{\iota_t}\ar@{}[dr]|{\square}\ar@{^{(}->}[d]_{\iota_{\cY_t}}& \cX_t\ar@{}[dr]|\square\ar[r] \ar@{^{(}->}[d]^{\iota_{\cX_t}}&\spec(k(t))\ar[d]\\
\cY\ar@{^{(}->}[r]^{\iota} & \cX \ar[r]&S},$$ the canonical morphisms 
\begin{equation}\label{Diag:Ample}\iota_{\cY_t}^*\cN_{\cY/\cX}\rightarrow \cN_{\cY_t/\cX_t} ,\;\; \iota_{\cY_{\bar t}}^*\cN_{\cY/\cX} \rightarrow \cN_{\cY_{\bar t}/\cX_{\bar t}} \end{equation}
are isomorphisms. Indeed, applying $\iota_{\cY_t}^*$ to the short exact sequence of locally free $\cO_{\cY}$-modules \cite[\S 6.3, Prop. 3.13]{LiuAG}
$$0\rightarrow \mathcal{C}_{\cY/\cX}\rightarrow \iota^*\Omega^1_{\cX|S}\rightarrow \Omega^1_{\cY|S}\rightarrow 0$$
  and using the canonical identifications  $$\iota_{\cY_t}^*\Omega^1_{\cY|S}\simeq \Omega^1_{\cY_t|k(t)},\;\; \iota_{\cY_t}^* \iota^*\Omega^1_{\cX|S}\simeq  \iota_t^*\iota_{\cX_t}^*\Omega^1_{\cX|S}\simeq \iota_t^* \Omega^1_{\cX_t|k(t)},$$
one gets the short exact sequence of locally free $\cO_{\cY_t}$-modules
$$0\rightarrow \iota_{\cY_t}^*\mathcal{C}_{\cY/\cX}\rightarrow \iota_t^*\Omega^1_{\cX_t|k(t)}\rightarrow \Omega^1_{\cY_t|k(t)}\rightarrow 0,$$
which   yields  
 $ \iota_{\cY_t}^*\mathcal{C}_{\cY/\cX}\simeq   \mathcal{C}_{\cY_t/\cX_t},$
whence the assertion, by dualizing. On the other hand, as $\cN_{\cY_{\bar \eta}/\cX_{\bar \eta}} (\simeq (\cN_{\cY/\cX})|_{\cY_{\bar \eta}})$ is ample, by fpqc descent of ampleness, $(\cN_{\cY/\cX})|_{\cY_{\eta}}$ is ample \cite[\href{https://stacks.math.columbia.edu/tag/0D2P}{Tag 0D2P}]{stacks-project}; the assertion thus follows from \cite[Prop. 4.4]{HartshorneAmpleness}.  \end{proof}

   \subsection{Sample of rigidity phenomena} We give here two examples of rigidity phenomena, building on  the classification results of  \cite{JKLM} and Corollary  \ref{Cor:MTBis} (2).\\
   
   \noindent  For an abelian variety $X$ over a field $K$ of characteristic $0$ and a closed subvariety $Y\subset X$, smooth, geometrically connected and of dimension $d\geq 2$ over $K$,  one says that $Y$ is:
 \begin{itemize}[leftmargin=*, parsep=0cm,itemsep=0.1cm,topsep=0.1cm] 
\item    a product if there exist closed subvarieties $Y_1,Y_2\subset X$, smooth over $K$ and of dimension $>0$, such that the sum map $+:Y_1\times Y_2\rightarrow X$ induces an isomorphism $+:Y_1\times Y_2\tilde{\rightarrow} Y$;
   \item a symmetric power of a curve if there is a closed smooth irreducible curve $C\hookrightarrow X$  such that the sum morphism Sym$^dC\rightarrow X$ is a closed embedding with image $Y$.
   \end{itemize}
   \noindent Note that if $K$ is algebraically closed and $L/K$ is a field extension, then $Y$ is a product (resp. a symmetric power of a curve) if and only if $Y\times_KL$ is. The only if assertion is straightforward and the if one follows from spreading out and specialization, using Hilbert Nullstellensatz.
 \begin{corollaire}\label{Cor:GeoProduct} Let $ \cX\rightarrow S$ an abelian scheme of relative dimension $g\geq 3$ and $ \cY\hookrightarrow \cX $ a closed subscheme, smooth and geometrically connected  over $S$. Assume $ \cY_{\bar \eta}\hookrightarrow \cX_{\bar \eta}$ has ample normal bundle  and   trivial stabilizer.  Then the set of  all  $s\in S$ such that $\cY_{\bar s}$ is a product is Zariski-dense in $S$ (if and) only if $\cY_{\bar \eta}$ is itself a product.
 \end{corollaire} 
\begin{proof}  From Lemma \ref{Lem:Geo} (1),    $\cP:=\iota_*\overline{\Q}_{\ell,\cY}[d]\in \Perv^{\ULA}(\cX/S)$ and from Lemma \ref{Lem:Geo} (2), up to replacing $S$ by a non-empty open subscheme, one may  assume that  for all $s\in S$, $\cY_{\bar s}\hookrightarrow \cX_{\bar s}$  has  ample normal bundle $\cN_{ \cY_{\bar s}/\cX_{\bar s}}$ and    trivial stabilizer.  The if assertion is by spreading out. To prove the only if assertion, observe first that, as $\cY_{\bar \eta}$ is smooth and irreducible  $\cP|_{\cX_{\bar \eta}}$ is simple - hence semisimple  \cite[Thm. 4.3.1 (ii)]{BBD} so that Corollary \ref{Cor:MTBis} (2) (b) applies.
The assertion thus follows from \cite[Thm. 6.1]{JKLM}, which asserts that, for any $t\in S$, $\cY_{\bar t}$ is  a product if and only if $G(\cP|_{\cY_{\bar t}})^{\circ\, \der}$ is not simple.  \end{proof}

\begin{remarque}\label{Rem:GeoProd}\textnormal{For  $\cY= \cX\rightarrow  S $, in general, it is not true that the set of all  $s\in S$ such that $\cX_{\bar s}$ is a product is Zariski-dense in $S$ (if and) only if $\cX_{\bar \eta}$ is itself a product. For instance, let  $k=\C$, $S=M_2$  the moduli space of genus $2$ smooth projective curves (with suitable level structures)
and $\cX:=\hbox{\rm Jac}(\cC|S)\rightarrow S$  the Jacobian of the universal genus $2$ curve $\cC\rightarrow S$.  Then the set  of all  points $s\in S$ such that  $\cX_{\bar s}$ is a product of two elliptic curves\footnote{Observe that if $+:X_1\times X_2 \tilde{\rightarrow}\cX_{\bar s}$ is a product then, necessarily, $X_1, X_2$ are translates of abelian subvarieties of $\cX_{\bar s}$.}  is supported on infinitely many irreducible curves $C_d\hookrightarrow S$. Let $S^{\hbox{\rm \tiny prod}}\subset S$ denote the Zariski-closure of the union of all $C_d$. Then the geometric generic fiber $\cX_{\bar \xi}$  over  the generic point $\xi$ of an irreducible component of $S^{\hbox{\rm \tiny prod}}$ of dimension $\geq 2$  is not a product of two elliptic curves.  See \cite{Kani}  and the references therein for details.}
\end{remarque}
 \begin{corollaire}\label{Cor:GeoSymProductCurve} Let $ \cX\rightarrow S$ an abelian scheme of relative dimension $g $ and $ \iota:\cY\hookrightarrow \cX $ a closed subscheme, smooth, geometrically connected  and of relative dimension $d<\frac{g-1}{2}$ over $S$. Assume $ \cY_{\bar \eta}\hookrightarrow \cX_{\bar \eta}$ has ample normal bundle and trivial stabilizer.  Then the set of  all  $s\in S$ such that $\cY_{\bar s}$ is a symmetric power of a curve  is Zariski-dense in $S$ (if and) only if $\cY_{\bar \eta}$ is itself a symmetric power of a curve. 
 \end{corollaire} 
\begin{proof}  The argument is similar to the one for Corollary \ref{Cor:GeoProduct}. If $d=1$ there is nothing to prove so that we may assume $d\geq 2$. Again,   $\cP:=\iota_*\overline{\Q}_{\ell,\cY}[d]$ lies in $\Perv^{\ULA}(\cX/S)$ with $\cP|_{\cX_{\bar \eta}}$  simple - hence semisimple, and,  up to replacing $S$ by a non-empty open subscheme, one may  assume that  for all $s\in S$, $\cY_{\bar s}\hookrightarrow \cX_{\bar s}$  has  ample normal bundle $\cN_{ \cY_{\bar s}/\cX_{\bar s}}$ and    trivial stabilizer.   The if assertion is by spreading out. To prove the only if assertion, let $r:=\chi(\cX_{\bar t},\cP)=(-1)^d\chi(\cY_{\bar t})$ denote the Euler-Poincar\'e characteristic of $\cP|_{\cX_{\bar t}}$ for one (equivalently every) $t\in S$.  Assume that for some  $s\in S$, $\cY_{\bar s}\simeq$Sym$^d(C)$ is   a symmetric power of a curve. Then from \cite[Lem. 7.2]{JKLM}, $G(\cP|_{\cY_{\bar s}})^{\circ\, \der}$ acting on $\omega(\cP|_{\cX_{\bar s}})$ identifies with the image of SL$_{n\overline{\Q}_\ell}$ acting on a wedge power $\wedge^d$Std$_n$ of the standard representation of  SL$_{n\overline{\Q}_\ell}$ with $n=-\chi(C,\overline{\Q}_\ell)=2g_C-2$ (where $g_C$ denotes the genus of $C$).  Note that, as $r=\binom{n}{d}$, $2g_C-2=n=:n(r,d)=$ is uniquely determined by $r$ and $d$ hence is independent of $s\in S$. Furthermore, as $W_d(C)\subset \hbox{\rm Alb}(C)$ is then automatically smooth,  it follows from Riemann's singularity theorem (\textit{e.g.}  \cite[p. 344]{GriffithHarris}) that  $C$ has gonality $\geq d+1$ hence genus $g_C\geq 2d-1$.    As $d\geq 2$, this imposes $g_C\geq 3$ hence $n(r,d)\geq 4$, whence,  $2\leq d\leq \frac{n(r,d)}{2}$. Under this numerical condition, it   follows from  \cite[Thm. 7.3]{JKLM} that, for any $t\in S$, $\cY_{\bar t}\simeq$Sym$^d(C)$ is   a symmetric power of a curve   if and only if $G(\cP|_{\cY_{\bar t}})^{\circ\, \der}$ acting on $\omega(\cP|_{\cX_{\bar t}})$ identifies with the image of SL$_{n\overline{\Q}_\ell}$ acting on a wedge power $\wedge^d$Std$_{n(r,d)}$ of the standard representation of  SL$_{n(r,d)\overline{\Q}_\ell}$. The assertion thus follows, again,  from Corollary \ref{Cor:MTBis} (2) (b).
\end{proof}
\begin{remarque}\label{Rem:GeoSymProductCurve}\textnormal{Using \cite[Thm. 6.1]{KM} instead of \cite[Thm. 7.3]{JKLM}, one could probably relax the assumption that $\cY\rightarrow S$ is smooth.}
\end{remarque}

\noindent \begin{tabular}[t]{l}
\textit{anna.cadoret@imj-prg.fr}\\
IMJ-PRG,  
Sorbonne Universit\'e,\\
 4, place Jussieu, 75252 PARIS Cedex 05, FRANCE\\
\\
\textit{hliu@imj-prg.fr}\\
IRMA, Universit\'e de Strasbourg,\\
7, rue Ren\'e Descartes, F-67084 Strasbourg Cedex, France
\end{tabular}\\
\end{document}